\documentclass[a4paper,12pt,reqno]{amsart}
\usepackage[T1]{fontenc}
\usepackage[utf8]{inputenc}
\usepackage{anyfontsize}
\usepackage[english]{babel}
\usepackage{stmaryrd}
\usepackage{amssymb}

\usepackage[%
	left=3.5cm,       
	right=3.5cm,      
	top=3.5cm,        
	bottom=3.5cm,     
	heightrounded,    
	bindingoffset=0mm 
]{geometry}

\usepackage[dvipsnames]{xcolor}
\usepackage{graphicx}
\usepackage{tikz}
\usetikzlibrary{calc}
\usetikzlibrary{angles}
\usetikzlibrary{arrows.meta}
\usetikzlibrary{decorations.pathreplacing}
\usepackage{flafter}
\usepackage{subcaption}

\usepackage{xfrac}
\usepackage{functan}
\usepackage{mathtools}
\usepackage{mathrsfs}

\numberwithin{equation}{section}
\allowdisplaybreaks

\newcommand{\R}{\mathbb{R}}
\newcommand{\Sf}{\mathbb{S}}
 \newcommand{\e}{\varepsilon}
\newcommand{\di}{\mathrm{d}}
\newcommand{\K}{\mathcal{K}}
\newcommand{\BK}{\mathcal{B}\mathcal{K}}
\newcommand{\Hau}{\mathcal{H}}

\DeclareMathOperator{\dist}{dist}
\DeclareMathOperator{\Per}{Per}
\DeclareMathOperator{\diam}{diam}

\DeclareMathOperator{\vspan}{span}
\DeclareMathOperator{\interior}{int}
\DeclareMathOperator{\coh}{coh}

\definecolor{grey}{rgb}{.7,.7,.7}
\definecolor{evidGP}{rgb}{.8,0,.8}

\usepackage{enumitem}

\usepackage[pagewise, mathlines]{lineno}  

\usepackage{etoolbox} 

 \newcommand*\linenomathpatch[1]{%
   \expandafter\pretocmd\csname #1\endcsname {\linenomath}{}{}%
   \expandafter\pretocmd\csname #1*\endcsname{\linenomath}{}{}%
   \expandafter\apptocmd\csname end#1\endcsname {\endlinenomath}{}{}%
   \expandafter\apptocmd\csname end#1*\endcsname{\endlinenomath}{}{}%
 }
\newcommand*\linenomathpatchAMS[1]{%
  \expandafter\pretocmd\csname #1\endcsname {\linenomathAMS}{}{}%
  \expandafter\pretocmd\csname #1*\endcsname{\linenomathAMS}{}{}%
  \expandafter\apptocmd\csname end#1\endcsname {\endlinenomath}{}{}%
  \expandafter\apptocmd\csname end#1*\endcsname{\endlinenomath}{}{}%
}

\expandafter\ifx\linenomath\linenomathWithnumbers
  \let\linenomathAMS\linenomathWithnumbers
  \patchcmd\linenomathAMS{\advance\postdisplaypenalty\linenopenalty}{}{}{}
\else
  \let\linenomathAMS\linenomathNonumbers
\fi

\linenomathpatch{equation} 
\linenomathpatchAMS{gather}
\linenomathpatchAMS{multline}
\linenomathpatchAMS{align}
\linenomathpatchAMS{alignat}
\linenomathpatchAMS{flalign}

\usepackage[
colorlinks=true,
urlcolor=teal,
citecolor=magenta,
linkcolor=teal,
linktocpage,
pdfpagelabels,
bookmarksnumbered,
bookmarksopen,
breaklinks=true]
{hyperref}

\usepackage[capitalize,nameinlink,noabbrev]
{cleveref}

\theoremstyle{plain}
	\newtheorem{theorem}{Theorem}[section]
    \newtheorem{lemma}[theorem]{Lemma}
    \AddToHook{env/lemma/begin}{\crefalias{theorem}{lemma}}
    \crefname{lemma}{Lemma}{Lemmas}
	\newtheorem{proposition}[theorem]{Proposition}
    \AddToHook{env/proposition/begin}{\crefalias{theorem}{proposition}}
    \crefname{proposition}{Proposition}{Propositions}
	\newtheorem{corollary}[theorem]{Corollary}
    \AddToHook{env/corollary/begin}{\crefalias{theorem}{corollary}}
    \crefname{corollary}{Corollary}{Corollaries}

\theoremstyle{definition}
	\newtheorem{definition}[theorem]{Definition}
    \AddToHook{env/definition/begin}{\crefalias{theorem}{definition}}
    \crefname{definition}{Definition}{Definitions}
	\newtheorem{example}[theorem]{Example}
    \AddToHook{env/example/begin}{\crefalias{theorem}{example}}
    \crefname{example}{Example}{Examples}
	\newtheorem{remark}[theorem]{Remark}
    \AddToHook{env/remark/begin}{\crefalias{theorem}{remark}}
    \crefname{remark}{Remark}{Remarks}
    
\title[Convex sets with vanishing relative width]{Convex sets with vanishing relative width and the reverse Cheeger inequality}

\author[N.~Cont]{Nicolò Cont}
\address[N.~Cont]{Department of Mathematics and Data Science, Vrije Universiteit Brussel, Pleinlaan 2,
BE-1050 Elsene (Belgium)}
\email{nicolo.cont@vub.be}

\author[G.P.~Leonardi]{Gian Paolo Leonardi}
\address[G.P.~Leonardi]{Dipartimento di Matematica, Universit\`a di Trento, via Sommarive 14, IT-38123 Povo--Trento (Italy)}
\email{gianpaolo.leonardi@unitn.it}

\author[G.~Saracco]{Giorgio Saracco}
\address[G.~Saracco]{
Dipartimento di Matematica e Informatica, Università di Ferrara, via Machiavelli 30, IT-44121 Ferrara (Italy)
}
\email{giorgio.saracco@unife.it}

\keywords{Cheeger constant, reverse Cheeger inequality, principal widths, convex bodies.}

\subjclass[2020]{Primary 49Q20. Secondary 52A20, 35P15}

\begin{document}


\begin{abstract}
We study the reverse Cheeger inequality, which bounds from above the ratio of the first Dirichlet Laplacian eigenvalue to the square of the Cheeger constant. We first extend this inequality from convex sets to a broader class. We then analyze maximizing sequences among convex bodies. To quantify domain collapse, we introduce the notion of principal widths. For three-dimensional convex bodies, we prove that if the ratio of the first principal width to the second principal width vanishes along a sequence of convex bodies, then such sequence is maximizing. Finally, in arbitrary dimensions, we prove that the same property holds for the class of rhomboid-like sets.
\end{abstract}

\maketitle


\section{Introduction}

Given a nonempty, bounded, and open set $\Omega\subset\R^N$, it is classically known that
\begin{equation}
\label{eq:Cheeineq}
J(\Omega)
=
\frac{\lambda_1(\Omega)}{h(\Omega)^2}
\ge
\frac{1}{4}\,,
\end{equation}
where $\lambda_1(\Omega)$ denotes the first eigenvalue of the Dirichlet Laplacian and $h(\Omega)$ the Cheeger constant of $\Omega$. Inequality~\eqref{eq:Cheeineq} is the celebrated \textit{Cheeger inequality}, a scale-invariant inequality proved for closed Riemannian manifolds by Cheeger in~\cite{Che70} (see also~\cite{Gri99} for the proof in the Euclidean case, using a capacitary inequality by Maz'ya~\cite{Maz62a}). A reverse form of~\eqref{eq:Cheeineq} was later obtained by Buser~\cite{Bus82} in the context of Riemannian manifolds. Specifically, when $M$ is a closed, $N$-dimensional manifold with Ricci curvature bounded from below by $-(N-1)\delta^2$ for some $\delta\ge 0$, Buser's inequality reads as
\[
\lambda_1(M) 
\le 
2(N-1)\delta\, h(M)+10\, h(M)^2\,.
\]
The proof of the inequality heavily relies on Heintze--Karcher comparison theorem coupled with the boundedness of the mean curvature of an optimal, codimension-$1$ submanifold obtained from the variational definition of $h(M)$. This proof strategy, however, is no longer applicable to the Dirichlet Laplacian on manifolds with boundary -- and thus, in particular, to the case of Euclidean domains in which we are specifically interested. Indeed, the topology of $\Omega$ (and notably its degree of connectedness) plays a crucial role in ensuring an upper bound of $\lambda_1(\Omega)$ in terms of $h(\Omega)$, as shown in~\cite[Prop.~5.1]{BB24}.

All these functionals can also be defined on convex bodies, i.e., bounded, closed, and convex sets with nonempty interior (see \cref{sec:preliminaries} for details).
In this context, Parini in~\cite[Prop.~4.1]{Par17} observed the validity of the inequality
\begin{equation}
\label{eq:Buser}
J(\Omega)
<
\frac{\pi^2}{4}\,,
\end{equation}
which we refer to here as the \textit{reverse Cheeger inequality} for the Dirichlet Laplacian. This is a direct consequence of a classical result by P\'olya~\cite{Pol60}. A closer inspection of P\'olya's proof reveals that a more general class of sets satisfying~\eqref{eq:Buser} is given by those whose Cheeger sets are \textit{perimeter-regular}. 
A perimeter-regular set is an open and bounded set with finite perimeter whose inner parallel sets have a perimeter strictly smaller than that of the full set (see \cref{sec:perimeter_regularity} for precise details).
Several extensions and variants of this type of spectral inequality can be found in the literature. In~\cite{Bra20}, a Buser-type inequality is derived for the ratio between the first Dirichlet eigenvalue of the $p$-Laplacian and the $p$-power of the Cheeger constant. In~\cite{BBP23,BBV25}, more general ratios involving first eigenvalues of $p$- and $q$-Laplacians are considered, together with the corresponding maximization problem among open sets (see in particular \cite[Thm.~2.9]{BBP23}, showing that the supremum is $+\infty$ and therefore cannot be attained on bounded open sets, but is realized by some smooth unbounded set, see \cite[Thm.~1.1]{BBV25}). 

Furthermore, the upper bound in~\eqref{eq:Buser} is asymptotically sharp. For instance, the sequence of parallelepipeds $P_k = (0,1)\times(0,k)^{N-1}$ satisfies $J(P_k)\to \pi^2/4$, see \cref{subs:maxseq}.
In dimension $N=2$, the characterization of sequences of convex bodies saturating~\eqref{eq:Buser} is even more stringent: they are precisely those of convex bodies which, up to rescaling, maintain a fixed minimal width while their diameters diverge (see \cref{prop:parini-saturate}). A similar characterization fails in higher dimensions, as illustrated by the Cartesian product of an $(N-1)$-dimensional unit ball with a sequence of intervals of diverging length, see \cref{ex:cylinders}.
To the best of our knowledge, the problem of fully characterizing such maximizing sequences is completely open in $\R^N,\ N\ge 3$.

Our main contribution is a partial characterization of the maximizing sequences for~\eqref{eq:Buser} among convex sets in dimension $N=3$, which is based on the notion of \emph{relative width} $w^{\mathrm{rel}}(\Omega)$ of a convex body $\Omega\subset\R^N$, see \cref{def:relative_width}. In \cref{thm:3-dim-reverse-cheeger}, we show that any sequence $(\Omega_k)_k$ of convex bodies in $\R^3$, for which $w^{(\mathrm{rel})}(\Omega_k)\to 0$, is a maximizing sequence for~\eqref{eq:Buser}.  
We also conjecture that these are the only possible maximizing sequences among convex bodies. Specifically, we expect that no sequence of convex bodies that are uniformly bounded with respect to at least two orthogonal directions (such as the cylinders in \cref{ex:cylinders}) can be  maximizing when $N\ge 3$.

The proof of \cref{thm:3-dim-reverse-cheeger} consists of showing that each $\Omega_k$, rescaled so that its minimal width is $1$, must contain---up to an isometry---a parallelepiped $R_k$ of the form
\[
R_k = [0,a_k]\times [0,b_k]^2\,,
\]
where $(a_k)_k$ and $(b_k)_k$ are sequences of positive real numbers satisfying $a_k\to 1$ and $b_k\to+\infty$.
This inner parallelepiped is used to estimate $h(\Omega_k)$ from above, whereas $\lambda_1(\Omega_k)$ is estimated from below by means of a sequence of parallelepipeds $P_k$ containing $\Omega_k$.

Moreover, in \cref{sec:rhomboid} we prove that a slightly stronger assumption on the sequence $(\Omega_k)_k$ is sufficient to guarantee that it saturates the reverse Cheeger inequality in any dimension.

The present paper is structured as follows.
\begin{itemize}
    \item In \cref{sec:preliminaries}, we recall some standard notions and known results, as well as the notation that will be used later on.
    \item In \cref{sec:perimeter_regularity} and \cref{sec:reverse_Cheeger}, we give the definition of perimeter-regularity of a Borel set in $\R^N$, supplementing it with some examples, and we prove that \eqref{eq:Buser} holds for every open and bounded set possessing a perimeter-regular Cheeger set.
    \item In \cref{subs:maxseq}, we provide an example of a sequence of sets that saturates \eqref{eq:Buser}, and we present a known characterization of such sequences in the case $N=2$.
    We also show that this characterization fails for every $N\ge3$, see \cref{ex:cylinders}.
    \item In \cref{sec:rel_width}, we introduce the notions of principal widths and relative width of a convex set, and show some related properties that will play a crucial role in the proof of our main result.
    \item Finally, in \cref{sec:maximizing_sequences}, we prove a sufficient condition for a sequence of convex sets to saturate \eqref{eq:Buser} in the case $N=3$, and we show that a slightly stronger assumption on this sequence allows us to extend this proof to every dimension $N\ge3$.
\end{itemize}

\section{Preliminaries and notation}
\label{sec:preliminaries}

In this section, we recall some well-known geometric definitions and properties, which will be used throughout the paper.

Let $\Omega\subset\R^N$ be a nonempty, bounded, and open set.
Its \emph{Cheeger constant} is defined as
\begin{equation}
\label{eq:cheeger_constant_def}
h(\Omega)=\inf \left\{\, \frac{\Per(E)}{|E|}\,:\, E\subset\Omega,\, |E|>0 \,\right\},
\end{equation}
where $\Per(-)$ denotes the variational perimeter, see~\cite{Mag12} for the definition.
Any subset $C$ attaining the infimum is called a Cheeger set of $\Omega$.

Here we list some well-known properties of the Cheeger constant:
\begin{enumerate}[label=(\roman*)]
    \item
    \label{item:cheeger-isometrie}
    $h(\Omega)$ is invariant by isometries of $\Omega$ in $\R^N$;
    \item
    \label{item:cheeger-scaling}
    (scaling) for every $t>0$, it holds that $h(t\Omega)=t^{-1}h(\Omega)$, where
    \[
        t\Omega=\{\,x\in\R^N : t^{-1}x\in\Omega\,\};
    \]
    \item
    \label{item:cheeger-monotonicity}
    (monotonicity) for every $\Sigma\subset\Omega$, it holds that $h(\Omega)\le h(\Sigma)$;
    \item
    \label{item:cheeger-existence}
    every bounded and open set $\Omega$ admits a Cheeger set, which is unique if we assume $\Omega$ to be convex;
    \item
    \label{item:cheeger-lip-bdd}
    if $\partial\Omega$ is Lipschitz, one can equivalently define $h(\Omega)$ by taking the infimum in~\eqref{eq:cheeger_constant_def} among sets $E\subset\overline{\Omega}$.
\end{enumerate}

Property~\ref{item:cheeger-isometrie} is trivial from the definition of $h(\Omega)$; properties~\ref{item:cheeger-scaling} and~\ref{item:cheeger-monotonicity} are proved in~\cite[Prop.~3.5]{Leo15}, as well as the existence of a Cheeger set mentioned in~\ref{item:cheeger-existence}; the uniqueness of such a set when $\Omega$ is convex is proved in~\cite[Thm.~1]{AC09}; finally, point~\ref{item:cheeger-lip-bdd} holds for sets enjoying even a weaker regularity than Lipschitz, see~\cite[Prop.~3.4]{LMR24}.

For every $\Omega$ as taken before, we define the functional
\[
J(\Omega)=\frac{\lambda_1(\Omega)}{h(\Omega)^2}\,,
\]
where $\lambda_1(\Omega)$ denotes the first Dirichlet eigenvalue of the Laplacian of $\Omega$, characterized as the infimum of its Rayleigh quotients, \emph{i.e.},
\[
\lambda_1(\Omega)=\inf\left\{\,\frac{\|\nabla v\|^2_2}{\| v\|^2_2}\,:\, v\in H^1_0(\Omega) \setminus\{ 0 \}\,\right\}.
\]
Thanks to the scaling property of $h(-)$ mentioned above, one can easily verify that $J(-)$ is invariant by isometries and scalings in $\R^N$.

Further, by virtue of property~\ref{item:cheeger-lip-bdd} of the Cheeger constant, and the fact that for Lipschitz sets one has
\[
H^1_0(\Omega)
=
\{u\in H^1(\R^N)\,:\, u=0\quad \text{a.e. on \(\R^N\setminus\Omega\)}\},
\]
it is convenient to extend the functionals $\lambda_1(-),\ h(-)$, and $J(-)$ to the closure $\overline{\Omega}$ of every open set $\Omega$ with Lipschitz boundary. 


The \emph{inradius} of a Borel set $E\subset\R^N$, which we denote by $\rho(E)$, is defined as
\[
\rho(E) = \sup \{\,\dist(x, \partial E)\,:\, x\in E\,\}.
\]
%
If $\rho(E)>0$, then for every $t$ such that $0<t<\rho(E)$ we define the \emph{inner parallel set} of $E$ at distance $t$ as the set
    \[
        E^t=\{\,x\in E:\dist(x,\partial E)\ge t\,\},
    \]
    which is a closed set contained in $E$. 

For every positive integer $N$, we denote by $\mathcal{K}_N$ the family of all the convex and closed sets $\Omega\subset\mathbb{R}^N$ with $\interior(\Omega)\ne\emptyset$.
Moreover, we denote by $\BK_N\subset\K_N$ the subfamily of all the bounded sets in $\K_N$, which coincides with the family of convex bodies in $\R^N$.

\section{Reverse Cheeger inequality}

As mentioned before, in this section we give the full derivation of the reverse Cheeger inequality~\eqref{eq:Buser}, thus providing the reader with the proper context needed in the rest of the paper.

\subsection{Perimeter-regularity}
\label{sec:perimeter_regularity}

We begin by giving the definition of \emph{perimeter-regularity}, which is crucial to the proof, and providing some notable examples of perimeter-regular sets.

\begin{definition}[Perimeter-regular set]
    Let $E\subset\R^N$ be a Borel set with $\rho(E)>0$.
    We say it is \emph{perimeter-regular} if
    \[
        \Per(E^t)\le\Per(E)
    \]
    for almost every $t$ such that $0<t<\rho(E)$, and \emph{strictly perimeter-regular} if the strict inequality holds.
\end{definition}

\begin{example}[Convex sets]
\label{ex:examples-convex}
Any closed convex set is strictly perimeter-regular by Steiner's formula~\cite{Ste13}, see also~\cite[Sect.~4.2]{Sch14}. The inverse is in general not true, see \cref{ex:examples-2d-k-connected}.
\end{example}

\begin{example}[$2$-dimensional $k$-connected sets]
\label{ex:examples-2d-k-connected}
A \(2\)-dimensional closed set \(E\) is called \(k\)-connected if its complement has \(k\) connected components; see \cite[Def.~1.1]{BB24}. We note that, if $E$ is path-connected, this definition corresponds to having genus $k-1$. If $E$ is \(k\)-connected with \(k=1,2\), it is perimeter-regular (strictly, if $k=1$), while for $k>2$, $E$ is not perimeter-regular; see~\cite[Thm.~II]{Mak59}. 
%

%
%
\end{example}

\begin{example}[Spherical shells]
\label{ex:examples-shells}
Let $S_{R,r}(x)\subset\mathbb{R}^N$ denote the spherical shell $B_R(x)\setminus \overline{B}_r(x)$ with $0<r<R$. Clearly, $(S_{R,r}(x))^t=S_{R-t, r+t}(x)$, and thus
\[
\Per((S_{R,r}(x))^t) = N\omega_N \big((R-t)^{N-1}+(r+t)^{N-1}\big) \qquad \forall t\,:\, 0 < t < \frac{R-r}{2}.
\]
Since $t\mapsto \Per((S_{R,r}(x))^t)$ has negative derivative in the interval $(0, (R-r)/2)$, the spherical shell $S_{R,r}(x)$ is strictly perimeter-regular.

\end{example}

\begin{example}[Tori]
\label{ex:examples-tori}
Let $\mathbb{T}\subset\mathbb{R}^3$ be a solid torus. Without loss of generality, assume $\mathbb{T}$ to be given by the rotation of the two-dimensional ball $B^2_r((R,0,0))$, lying in the plane $xz$, about the $z$-axis, with $0<r<R$.
Clearly, $\mathbb{T}^t$ is the solid torus given by the rotation of $B^2_{r-t}((R,0,0))$ about the same axis. By Pappus--Guldinus Theorem, we have
\[
\Per(\mathbb{T}^t) = 4\pi^2 R(r-t),
\]
which, as a function of $t\in(0,r)$, is strictly decreasing. Therefore, the torus $\mathbb{T}$ is strictly perimeter-regular.
\end{example}

Let $\Sigma$ be a smooth, compact, oriented surface without boundary in $\mathbb{R}^3$. By Weyl's tube area formula~\cite{Wey39} (refer also to Federer's curvature measures~\cite{Fed59}), for $\varepsilon$ sufficiently small,
\begin{align}
\label{eq:weyl}
\Hau^2(\Sigma_\varepsilon) &= \Hau^2(\Sigma) +\varepsilon \int_{\Sigma} H\,\mathrm{d}\Hau^{2} + \varepsilon^2 \int_{\Sigma} K\,\mathrm{d}\Hau^{2} \nonumber \\
&= \Hau^2(\Sigma) +\varepsilon \int_{\Sigma} H\,\mathrm{d}\Hau^{2} + 4\pi(1-g(\Sigma))\varepsilon^2,    
\end{align}
where $H$ denotes the mean curvature of $\Sigma$, $K$ the Gaussian curvature of $\Sigma$, $\Sigma_\varepsilon$ the parallel surface at \emph{signed distance} $\varepsilon$ from $\Sigma$, and $g(\Sigma)$ the genus of $\Sigma$ (the second identity follows from Gauss--Bonnet Theorem). This is a consequence of the general Weyl's tube formula expressing the $N$-volume of the $\e$-neighbourhood of $\Sigma$ as a polynomial of degree $N$ in $\e$ with coefficients depending on the principal curvatures of $\Sigma$. 

The next example is a direct consequence of this result.

\begin{example}[Toroidal sets] 
\label{ex:examples-3d-surfaces}
Let $\mathbb{T}\subset\mathbb{R}^3$ be a solid toroidal set, that is, a set such that $\partial\mathbb{T}$ is a compact, non-overlapping, curved tube, as defined in~\cite{KLV19}, and let $\gamma$ be the associated curve.
By Weyl's formula applied to $\Sigma=\partial\mathbb{T}$ we know that
\[
\Per(\mathbb{T}^t)=\Per(\mathbb{T})-t\int_{\partial\mathbb{T}} H\,\mathrm{d}\Hau^2
\]
because the genus of $\partial\mathbb{T}$ is $1$.
If we denote by $(\mathbf{t},\mathbf{n},\mathbf{b})$ the Frenet's basis associated to $\gamma$, then $\partial\mathbb{T}$ can be parametrized by
\[
    F(s,\theta)=\gamma(s)+r(\cos(\theta)\mathbf{n}(s)+\sin(\theta)\mathbf{b}(s))
\]
where $r$ denotes the distance between $\partial\mathbb{T}$ and $\gamma$.
By computing the principal directions of this surface and applying Frenet's formulas, we get the area element
\[
    \di A =r(1-r\kappa(s)\cos(\theta))\,\di\theta\,\di s
\]
where $\kappa$ denotes the curvature of $\gamma$.
Moreover, the two principal curvatures of $\partial\mathbb{T}$ are
\[
    k_1=\frac{1}{r} \qquad \text{and} \qquad k_2=-\frac{\kappa(s)\cos(\theta)}{1-r\kappa(s)\cos(\theta)},
\]
and thus
\begin{align*}
    \int_{\partial\mathbb{T}} H\,\mathrm{d}\Hau^2&= \int_0^L \int_0^{2\pi}\left(\frac{1}{r}-\frac{\kappa(s)\cos(\theta)}{1-r\kappa(s)\cos(\theta)}\right)r(1-r\kappa(s)\cos(\theta))\,\di\theta\,\di s \\
    &=\int_0^L \int_0^{2\pi}(1-2r\kappa(s)\cos(\theta))\,\di\theta\,\di s \\
    &=2\pi L,
\end{align*}
where $L$ denotes the length of the curve $\gamma$.
In conclusion, we find that $\Per(\mathbb{T}^t)=\Per(\mathbb{T})-2\pi Lt<\Per(\mathbb{T})$ for every $t>0$, which means that $\mathbb{T}$ is strictly perimeter-regular.

This agrees with~\cref{ex:examples-tori}, since in that case we have $L=2\pi R$ and $\Per(\mathbb{T})=4\pi^2Rr$.
\end{example}

\subsection{Proof of the reverse Cheeger inequality}
\label{sec:reverse_Cheeger}

We reproduce here the proof of a result from~\cite[Ineq. 2]{Pol60}. There, the author proved it under the assumption that $A$ is an open, bounded and convex set in the Euclidean plane. However, as highlighted by the following lemma, the same can actually be obtained for possibly unbounded, perimeter-regular open sets in any dimension.

\begin{lemma}
\label{lem:polya-ineq}
Let $A\subset\R^N$ be open and perimeter-regular. Then,
\begin{equation}
\label{ineq:polya}
  \lambda_1(A) |A|^2 
  \le
  \frac{\pi^2}{4} \Per(A)^2.
\end{equation}
Moreover, the inequality is strict if the set is strictly perimeter-regular.
\end{lemma}

\begin{proof}
Recalling the convention $0\cdot \infty = 0$, if $\Per(A) = +\infty$ there is nothing to prove. Then, the remaining proof is split into two cases. 
\medskip

\textit{Case $1$: $\Per(A)<+\infty$ and $|A|=+\infty$.} By the isoperimetric inequality in $\R^N$ we have $|A^c| = |\R^N\setminus A|<+\infty$, and hence by the perimeter-regularity hypothesis we infer that the set 
\[
(A^t)^c = \{x\in \R^N:\ \dist(x;A^c)< t\} = \R^N\setminus A^t
\]
satisfies
\begin{align*}
|(A^t)^c| & = |A^c| + \int_0^t \Per((A^s)^c)\, \mathrm{d}s = |A^c| + \int_0^t \Per(A^s)\, \mathrm{d}s \\
& \le |A^c| + t\Per(A) <+\infty\,.
\end{align*}
Let $\eta_R$ be a radial Lipschitz cutoff function constantly equal to $1$ on $B(0,R)$ and $0$ on $\R^N\setminus B(0,2R)$, with $|\nabla \eta|=1/R$ on $B(0,2R)\setminus \overline{B}(0,R)$. Define 
\[
u_R(x) = \dist(x;A^c) \eta_R(x)
\]
and note that $u_R\in H^1_0(A)$, with
\begin{align*}
\|u_R\|_2^2 &\ge t^2|B_R\cap A^t|
\ge t^2|B_R|/2 
\end{align*}
for all $t>0$ and for $R$ large enough (depending on $t$), while at the same time we have 
\begin{align*}
\|\nabla u_R\|_2^2 &\le \int_{A \cap (B_{2R}\setminus B_R)}| \eta\nabla \dist(x;\partial A) + \dist(x;\partial A)\nabla\eta|^2\, \mathrm{d}x\\
&\le \int_{A \cap (B_{2R}\setminus B_R)}(1 + \dist(x;\partial A)/R)^2\, \mathrm{d}x\\
&\le \int_{A \cap (B_{2R}\setminus B_R)}(1 + \diam(B_{2R})/R)^2\, \mathrm{d}x\\
&= 25|A \cap (B_{2R}\setminus B_R)|\\
&\le 25 |B_{2R}| = 25\cdot 2^N |B_R|\,.
\end{align*}
Then, if we fix $t^2\gg 25\cdot 2^{N+1}$  and choose $R$ large enough, we can conclude that $\|\nabla u_R\|_2 / \|u_R\|_2$ can be made arbitrarily small, which shows that $\lambda_1(A) = 0$ and finally proves \eqref{ineq:polya} in this case.
\medskip

\textit{Case $2$: $|A|<+\infty$ and $\Per(A)<+\infty$.} 
Since~\eqref{ineq:polya} is scale-invariant, we can assume without loss of generality that $\rho(A)=1$. Let now $M,L\colon(0,1)\to\R$ be defined as
\[
    M(t)=|A^t| \qquad \text{and} \qquad L(t)=\Per(A^t).
\]
Notice that $M$ is Lipschitz-continuous, hence differentiable almost everywhere. In particular, by the coarea formula, see, \emph{e.g.},~\cite[Chap.~13]{Mag12}, we have that
\[
    M(t) = \int_t^{1} L(\eta)\,\di\eta,
\]
from which we infer that
\begin{equation}
\label{eq:relation-M&L}
    M'(t)=-L(t), \qquad \text{for a.e.\ $t\in (0,1)$.}
\end{equation}
Further, since by assumption $A$ is perimeter-regular, we have
\begin{equation}
\label{ineq:p-reg_lem-polya}
    L(t)\le\Per(A), \qquad \text{for a.e.\ $t\in (0,1)$.}
\end{equation}

We now define the Lipschitz function
\[
w(x) = \cos\left(\frac{\pi}{2|A|}\,M\big(\dist(x;\partial A)\big)\right).
\]
We claim that $w\in H^1_0(A)$, and thus its Rayleigh quotient $\|\nabla w\|_2^2/\|w\|_2^2$ can be used to estimate $\lambda_1(A)$ from above.

Applying Fubini--Tonelli, the change of variables
\begin{equation}
\label{eq:change_variables}
s = \frac{\pi}{2|A|}M(t),
\end{equation}
and \eqref{eq:relation-M&L}, we get
\begin{align}
\label{eq:norm_w}
\int_A |w(x)|^2\,\di x
&
=
\int_0^1 \left[ \cos\left(\frac{\pi}{2|A|}\,M(t)\right) \right]^2L(t)\,\di t
\nonumber
\\
&
=
\frac{2|A|}{\pi}\int_0^{\frac\pi2}\cos^2(s)\,\di s \nonumber
\\
&=
\frac{|A|}{2}<+\infty.
\end{align}
Hence, $w$ belongs to $L^2(A)$.

On the other hand,  since both $M(-)$ and $\dist(-;\partial A)$ are continuous and almost everywhere differentiable, so is $w$, and it holds
\[
\nabla w(x) = -\sin \left(\frac{\pi}{2|A|}\,M\big(\dist(x;\partial A)\big)\right) \frac{\pi}{2|A|} M'\big(\dist(x;\partial A)\big) \nabla \dist(x;\partial A).
\]
Hence, using Fubini--Tonelli, the fact that $|\nabla\dist(-;\partial A)|$ is almost everywhere $1$, the change of variables~\eqref{eq:change_variables}, and the equations~\eqref{eq:relation-M&L} and~\eqref{ineq:p-reg_lem-polya}, we get
\begin{align}
\label{eq:norm_grad_w}
\int_A |\nabla w(x)|^2\,\di x
&
=
\left(\frac{\pi}{2|A|}\right)^2 \int_0^1 \left[ \sin\left(\frac{\pi}{2|A|}\,M(t)\right) M'(t) \right]^2L(t)\,\di t
\nonumber
\\
&
=
\frac{\pi}{2|A|} \int_0^{\frac\pi2} \sin^2(s) \Big(M'\big(t(s)\big)\Big)^2\,\di s
\nonumber
\\
&
\le
\frac{\pi^2}{8|A|} \Per(A)^2<+\infty.
\end{align}
Thus, $w$ belongs to $H^1(A)$.
Finally, notice that $w$ vanishes on $\partial A$, and therefore it belongs to the closed subspace $H^1_0(A)$, as claimed.

Thus, owing to the variational characterization of the first Dirichlet eigenvalue, using~\eqref{eq:norm_w} and~\eqref{eq:norm_grad_w}, we can immediately conclude
\begin{equation}
\label{eq:final}
    \lambda_1(A)
    =
    \inf\left\{\,\frac{\|\nabla v\|^2_2}{\| v\|^2_2}\,:\, v\in H^1_0(A)\,\right\}
    \le 
    \frac{\|\nabla w\|^2_2}{\| w\|^2_2}
    \le
    \frac{\pi^2}{4|A|^2}\Per(A)^2.
\end{equation}

Note that if $A$ is strictly perimeter-regular, then inequality~\eqref{ineq:p-reg_lem-polya} is strict on a set of positive measure, and so is the one in~\eqref{eq:norm_grad_w} as well, which, in turn, produces a strict inequality in~\eqref{eq:final}.
\end{proof}

This lemma yields the following proposition, which shows that the reverse Cheeger inequality holds for every open and bounded set $\Omega$ that admits a perimeter-regular Cheeger set.
\begin{proposition}[Reverse Cheeger inequality]
\label{prop:rev_Cheeger}
Let $\Omega\subset\R^N$ be an open and bounded set, and let $C$ be a Cheeger set for $\Omega$.
If $C$ is perimeter-regular, then
\[
    J(\Omega)\le\frac{\pi^2}{4}.
\]
Moreover, if $C$ is strictly perimeter-regular then the inequality is strict.
\end{proposition}
\begin{proof}
By the monotonicity of $\lambda_1(-)$, we have $\lambda_1(\Omega)\le\lambda_1(C)$. Since $C$ is a local lambda-minimizer of the perimeter in $\Omega$, we have that $C$ is Lebesgue-equivalent to the set $C^{(1)}$ of its Lebesgue points inside $\Omega$, which is an open set such that $|C^{(1)}| = |C|$ and $\Per(C^{(1)})=\Per(C)$. Applying \cref{lem:polya-ineq} to $C^{(1)}$, we get
%
\[
    \lambda_1(\Omega)
    \le
    \lambda_1(C^{(1)})
    \le 
    \frac{\pi^2}{4}\, \frac{\Per(C^{(1)})^2}{|C^{(1)}|^2}
    = \frac{\pi^2}{4}\, \frac{\Per(C)^2}{|C|^2}
    =
    \frac{\pi^2}{4} h(\Omega)^2,
\]
which is equivalent to the desired inequality.
If $C$ is strictly perimeter-regular, the second inequality is strict.
\end{proof}

It now follows that the inequality holds in the planar setting for sets that are at most $2$-connected. This extends the classical reverse Cheeger inequality for planar convex sets and provides the sharp constant in~\cite[Thm.~5.4]{BB24} for $k=2$.

\begin{corollary}
\label{cor:rev_cheeger_dim2}
Let $N=2$ and $\Omega$ be an open, bounded and $k$-connected set with $k=1,2$. Then, the reverse Cheeger inequality holds (strictly if $k=1$).
\end{corollary}

\begin{proof}
Note that if $\Omega$ is $k$-connected, then any of its connected components is at most $k$-connected.  
Let $C$ be a Cheeger set of $\Omega$. We can assume that $C$ is connected, hence contained in one of the connected components of $\Omega$ (that we identify with $\Omega$ itself). We claim that $C$ is at most $k$-connected.

Assume by contradiction that $C$ is $j$-connected with $j>k$, and let $(\partial C)_i$ with $i=1,\dots,j$ be the connected components of its boundary, while $(\partial \Omega)_i$ with $i=1,\dots, k$ are the connected components of the boundary of $\Omega$. Up to relabeling, the set bounded by $(\partial C)_i$ contains the one bounded by $(\partial \Omega)_i$, for all $i=1,\dots,\bar{k}$ and for some $1\le\bar{k}\le k$. Then, if we let $E_i$ be the set bounded by $(\partial C)_i$ with $i=\bar{k}+1,\dots,j$, we can take 
\[
C'=C\cup E_{\bar{k}+1}\cup\dots\cup E_j,
\]
which is a subset of $\Omega$. It is immediate to see that $\frac{\Per(C')}{|C'|}<\frac{\Per(C)}{|C|}$, contradicting the assumption that $C$ is a Cheeger set.

Therefore, by \cref{ex:examples-2d-k-connected} we know that $C$ is perimeter-regular (strictly perimeter-regular if $k=1$), which concludes the proof by virtue of \cref{prop:rev_Cheeger}.
\end{proof}

\begin{remark}
Since spherical shells and toroidal sets are perimeter-regular and self-Cheeger (see \cite{KLV19}), by \cref{prop:rev_Cheeger} we infer that they also satisfy \eqref{eq:Buser}.
\end{remark}

\subsection{Maximizing sequences}
\label{subs:maxseq}

The strict reverse Cheeger inequality is optimal for convex sets, see, \emph{e.g.},~\cite[Prop.~4.1]{Par17} for the planar case and~\cite[Sect.~4]{Bra20} for the general $N$-dimensional case.
In fact, in these works it is shown that the sequence of $N$-dimensional parallelepipeds $(\Omega_n)_n$ defined as
\[
\Omega_n =(0,1)\times\left(-\frac{n}{2},\frac{n}{2}\right)^{N-1}
\]
is such that $J(\Omega_n)\to\pi^2/4$ as $n\to+\infty$.
In order to prove this fact, it suffices to notice that
\begin{equation}
\label{eq:h_parallelepiped}
h(\Omega_n)\le\frac{\Per(\Omega_n)}{|\Omega_n|}=\frac{2n^{N-1}+(2N-2)n^{N-2}}{n^{N-1}}=2+\frac{2N-2}{n},    
\end{equation}
while on the other hand, by separation of variables,
\begin{equation}
\label{eq:lambda_parallelepiped}
\lambda_1(\Omega_n)=\pi^2\left(1+\frac{N-1}{n^2}\right).
\end{equation}
Therefore, by combining \eqref{eq:h_parallelepiped} and \eqref{eq:lambda_parallelepiped} we find that
\[
\liminf_{n\to+\infty} J(\Omega_n)\ge\lim_{n\to+\infty}\pi^2\left(1+\frac{N-1}{n^2}\right)\left(2+\frac{2N-2}{n}\right)^{-2}=\frac{\pi^2}{4}\,,
\]
which together with \cref{prop:rev_Cheeger} implies that $J(\Omega_n)\to \pi^2/4$ as $n\to+\infty$.

The following proposition, proved in~\cite[Prop.~4.1]{Par17}, provides a full characterization for a sequence to be maximizing in dimension $N=2$.

\begin{proposition}
\label{prop:parini-saturate}
    Let $(\Omega_n)_n\subset\BK_2$ be a sequence with constant \emph{minimal width} $w(\Omega_n)$ (see \cref{subsec:principal_widths} for the definition of $w(-)$).
    Then, $(\Omega_n)_n$ is maximizing for the functional $J$ if and only if $\diam(\Omega_n)\to+\infty$.
\end{proposition}

In higher dimensions, the conclusion of \cref{prop:parini-saturate} is no longer true, as shown by the next example.

\begin{example}
\label{ex:cylinders}
Let $N\ge3$ and $R>0$, and consider the sequence $(\Omega_n)_n$ of cylinders in $\mathbb{R}^{N}$, where
\[
\Omega_n = B^{N-1}_{R}(0) \times (0, n).
\]
These sets have constant minimal width $w(\Omega_n)=2R$ for every $n$ sufficiently large, whereas the corresponding sequence of diameters is unbounded. Let us estimate both $h(\Omega_n)$ and $\lambda_1(\Omega_n)$.

By~\cite[Thm.~2.1]{PS25}, if $(S_n)_{n}$ is a sequence of $N$-dimensional cylinders of the form $S_n=\Sigma\times(0,n)$ for some $\Sigma\subset\R^{N-1}$, when $n\to+\infty$ it holds that \(h(S_n)\to h(\Sigma)\). Therefore, here we have
\begin{equation}
\label{eq:asympt-h}
    h(\Omega_n)
    \to h(B^{N-1}_{R}(0))
    = 
    \frac{N-1}{R}.
\end{equation}

On the other hand, by separation of variables, one can see that the first Dirichlet eigenvalue of $\Omega_n$ is
\[
    \lambda_{1}(\Omega_n)
    = j^2_{\frac{N-3}{2}, 1}\, R^{-2} + \frac{\pi^2}{n^2},
\]
where $j_{m,n}$ denotes the $n$-th positive zero of the Bessel function of order $m$. Hence, we have the asymptotic behavior
\begin{equation}
\label{eq:asympt-lambda1}
\lambda_1(\Omega_n)\to j^2_{\frac{N-3}{2}, 1} R^{-2}.
\end{equation}

Combining~\eqref{eq:asympt-h} and~\eqref{eq:asympt-lambda1} shows that
\[
    \lim_{n\to+\infty}J(\Omega_n)=\left(\frac{j_{\frac{N-3}{2},1}}{N-1}\right)^2<\frac{\pi^2}{4}, \qquad \forall N\ge3.
\]
Thus, $(\Omega_n)_n$ is not a maximizing sequence, despite having constant minimal width and diameter divergent to infinity.
\end{example}

Given a sequence $(\Omega_n)_n\subset\BK_N$ with these properties, each set $\Omega_n$ has two independent directions in $\R^N$ associated to $w^{(1)}(\Omega_n)$ and $\diam(\Omega_n)$, respectively.
If $N=2$, these two directions form a basis of the space, whereas in dimension $N>2$ there are extra $N-2$ independent directions along which, in general, one has no information on the behavior of the set.


\section{Principal widths of a convex set}
\label{sec:rel_width}

In this section, we introduce the concept of \emph{principal widths} of a convex set $\Omega\subset\R^N$ and we prove some important properties.
Further, we define the \emph{relative width} of $\Omega$, which will be crucial in the next section.

Here and throughout this section, unless explicitly stated, we will assume $\Omega\in\K_N$ for some $N\ge2$.

\subsection{Principal widths}
\label{subsec:principal_widths}

We begin by giving the definition of the \emph{principal widths} of $\Omega$, which are the tools we use to formalize the idea 

The \emph{support function of $\Omega$}, denoted\footnote{This is most commonly denoted by $h_\Omega(-)$, but we prefer to use as notation $\sigma_\Omega(-)$ in order to avoid any possible confusion with the Cheeger constant $h(-)$.} by  $\sigma_\Omega:\mathbb S^{N-1}\to \mathbb R\cup\{+\infty\}$, is defined as
\[
\sigma_\Omega(\nu)=\sup_{x\in \Omega}\, \langle x, \nu\rangle,
\]
and it can be extended by positive $1$-homogeneity to the punctured space $\mathbb{R}^N\setminus\{0\}$; for the basic properties of $\sigma_\Omega$ we refer to~\cite[Sect.~1.7.1]{Sch14}. We here just remark that $\sigma_\Omega$ is lower semicontinuous, being the supremum of continuous functions; if $\Omega$ is bounded, then $\sigma_\Omega$ is continuous, see, \emph{e.g.},~\cite[Cor.~2.2]{HW20}.

Given a direction $\nu$, we define the \emph{supporting half-space to $\Omega$ with outer unit normal $\nu$} as
\[
H_\nu^{-}(\Omega) 
= 
\{\,x\in\mathbb{R}^N\,:\, \langle x,\nu\rangle \le \sigma_\Omega(\nu)\,\}.
\]
If $\sigma_\Omega(\nu)$ is finite, then we can consider the boundary of this half-space,
\begin{equation}
\label{eq:supporting_hyperplane}
H_\nu(\Omega) 
= 
\{\,x\in\mathbb{R}^N\,:\, x\cdot\nu = \sigma_\Omega(\nu)\,\},
\end{equation}
called the \emph{supporting hyperplane to $\Omega$ with outer unit normal $\nu$}. 
Then, $-\sigma_\Omega(\nu)$ represents the signed distance from the origin of $H_\nu(\Omega)$.
In particular, this distance is positive if $H_\nu^-(\Omega)$ does not contain the origin and negative if it does.
Incidentally, we recall that
\begin{equation}
\label{eq:intersection-supp-hyperplanes}
\Omega = \bigcap_{\mathbb{S}^{N-1}} H^-_{\nu}(\Omega).
\end{equation}
Further, we define the \emph{support set of $\Omega$ with outer unit normal $\nu$} as
\begin{equation}
\label{eq:face}
F(\Omega, H_\nu(\Omega)) = \Omega\cap H_\nu(\Omega),
\end{equation}
that is, roughly speaking, the ``face'' of the set consisting of points whose outer unit normal is $\nu$.

The \emph{directional width of $\Omega$ in direction $\nu$}, denoted by $w_{\nu}(\Omega)$, is defined as
%
\[
w_\nu(\Omega)=\sigma_\Omega(\nu)+\sigma_\Omega(-\nu) \in [0,+\infty].
\]
%
If $w_\nu(\Omega)$ is finite, then it represents the distance between the two supporting hyperplanes to $\Omega$ orthogonal to the direction $\nu$, $H_{\nu}(\Omega)$ and $H_{-\nu}(\Omega)$. On the other hand, if $w_\nu(\Omega)$ is infinite, then at least one of the supporting half-spaces with outer unit normal either $\nu$ or $-\nu$ is the whole space $\mathbb{R}^N$.

Finally, the \emph{minimal width} or \emph{thickness} is the minimum of the directional widths among all directions, \emph{i.e.},
\begin{equation}
\label{eq:minimal_w}
w(\Omega)=\inf_{\nu\in\mathbb S^{N-1}}w_\nu(\Omega).
\end{equation}
We stress that this is a minimum by Weierstrass Theorem, since the function $\nu\mapsto \sigma_\Omega(\nu)$ is lower semicontinuous as mentioned before, so is $\nu\mapsto w_\nu(\Omega)$.

We now define the $N$ principal widths of a nonempty, closed, and convex set $\Omega$ as follows.

\begin{definition}[Principal widths]
\label{def:p-widths}
Let $\Omega\in\K_N$. We start by setting $w^{(1)}(\Omega) = w(\Omega)$ and
\[
V^{(1)}(\Omega) = \{u^{(1)}\in \Sf^{N-1}:\ w_{u^{(1)}}(\Omega) = w^{(1)}(\Omega)\}.
\]
For the next step, we define
\[
w^{(2)}(\Omega) = \inf \{w_\nu(\Omega):\ \nu \in \Sf^{N-1},\ \langle \nu, u^{(1)}\rangle = 0,\ u^{(1)}\in V^{(1)}(\Omega)\}
\]
and
\[
V^{(2)}(\Omega) = \{(u^{(1)},u^{(2)})\in V^{(1)}\times\Sf^{N-1} : \langle u^{(1)},u^{(2)}\rangle=0,\ w_{u^{(2)}}(\Omega) = w^{(2)}(\Omega)\} .
\]
Then, we proceed recursively as $i=3,\dots,N$, so that for each step we get a (possibly infinite) value $w^{(i)}(\Omega)$ and a set of $i$-uples $V^{(i)}\subset(\Sf^{N-1})^i$ such that
\[
    \forall (u^{(1)},\dots,u^{(i)})\in V^{(i)}, \quad \forall j,k=1,\dots,i \qquad \begin{cases}
        \langle u^{(j)} , u^{(k)} \rangle = \delta_{jk} \\
        w_{u^{(j)}}(\Omega)=w^{(j)}(\Omega) .
    \end{cases}
\]
Each value $w^{(i)}(\Omega)$ is called the \emph{$i$-th principal width of} $\Omega$.
Moreover, when we reach the last step $i=N$, every $N$-uple in the space $V^{(N)}(\Omega)$ is called a \emph{set of principal directions of} $\Omega$, denoted by $(\nu^{(1)}(\Omega),\dots,\nu^{(N)}(\Omega))$.

For the sake of notation, whenever the set $\Omega$ is fixed, we shall simply write $w^{(i)}$, $\nu^{(i)}$ and $V^{(i)}$, dropping the explicit dependence on $\Omega$.
\end{definition}

\begin{remark}
\label{rem:well-posedness}
The definitions of $V^{(i)}$ and of principal widths $w^{(i)}$ are canonical. Indeed, for every $i=2,\dots,N$ the infimum defining $w^{(i)}$ is always attained, even when it is $+\infty$.
\end{remark}

\begin{remark}[Principal directions]
\label{rmk:principal_directions}
Unlike principal widths, principal directions are not uniquely determined by the set $\Omega\in\K_N$ at hand (in other words, $V^{(N)}$ may contain more than one basis of principal directions).
For instance, if $\Omega$ is an $N$-dimensional ball, then every orthonormal basis of $\R^N$ is trivially a set of principal directions of $\Omega$.

On the other hand, not every vector $u\in\mathbb{S}^{N-1}$ that realizes a principal width $w^{(i)}$ is necessarily a principal direction.
For example, if we consider a set $\Omega\in\K_2$ like in \cref{img:multiple_infimum}, we can easily see that the vector $u^{(1)}=e_2$ belongs to $V^{(1)}$ but does not realize the infimum that defines $w^{(2)}$.
Indeed, the infimum
\[
w^{(2)}(\Omega) = \inf \{w_\nu(\Omega):\ \nu \in \Sf^{N-1},\ \langle \nu, u^{(1)}\rangle = 0,\, u^{(1)}\in V^{(1)}(\Omega)\}
\]
is attained by taking, for example, $u^{(1)}=(e_1+e_2)/\sqrt{2}$ and $\nu=(e_1-e_2)/\sqrt{2}$, whereas $u^{(1)}=e_2$ does not realize it, because with that choice we would get $\nu=\pm e_1$ and thus $w_\nu(\Omega)=\diam(\Omega)>w^{(2)}(\Omega)$.
\end{remark}

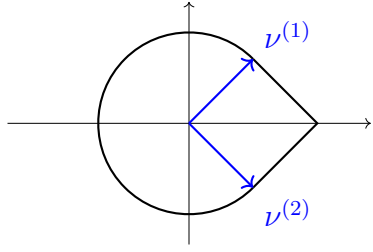
\begin{figure}[ht]
\centering
\begin{tikzpicture}[scale=0.4]
  \coordinate (O) at (-0.00, 0.00);
  \coordinate (K) at (-6.00, 0.00);
  \coordinate (L) at (6.00, 0.00);
  \coordinate (M) at (0.00, 4.00);
  \coordinate (N) at (0.00, -4.00);
  \coordinate (A) at ($(O) + (45:3cm)$);
  \coordinate (B) at ($(O) + (315:3cm)$);
  \coordinate (C) at ($(A) + (315:3cm)$);

  \draw[ultra thin, ->] (K) -- (L);
  \draw[ultra thin, ->] (N) -- (M);
  \draw[thick] (A) arc (45:315:3cm);
  \draw[thick] (A) -- (C) -- (B);
  \draw[thick, blue, ->] (O) -- (A) ;
  \draw[thick, blue, ->] (O) -- (B);
  \node[above right, blue] at (A) {$\nu^{(1)}$};
  \node[below right, blue] at (B) {$\nu^{(2)}$};
\end{tikzpicture}
\caption{The example considered in \cref{rmk:principal_directions}.}
\label{img:multiple_infimum}
\end{figure}

\begin{remark}[Canonical basis]
\label{rem:canonical_basis}
Since by definition every set of principal directions $(\nu^{(1)},\dots,\nu^{(N)})\in V^{(N)}$ is an orthonormal basis of $\mathbb{R}^N$, one can always assume up to an isometry that $\nu^{(i)} = e_i$.
\end{remark}

We collect some of the properties of the principal widths in the next proposition, which is essentially straightforward from the definition.

\begin{proposition}
\label{prop:properties_p-widths}
Let $\Omega\in\K_N$. We have the following:
\begin{enumerate}[label=(\roman*)]
    \item $w^{(i)} \le w^{(i+1)}$, for all $i=1,\dots,N-1$;
    \item $\Omega$ is contained in a (possibly unbounded) $N$-dimensional parallelepiped with edges directed as $\nu^{(i)}$ and of lengths $w^{(i)}$;
    \item $\Omega$ is bounded if and only if $w^{(N)}<+\infty$.
\end{enumerate}
\end{proposition}

\begin{proof}
Point~(i) follows from the definition, since the space over which we take the infimum is restricted at each step.

Point~(ii) follows from~\eqref{eq:intersection-supp-hyperplanes}, restricting the intersection to the pairwise orthogonal half-spaces in directions $\pm \nu^{(i)}$.

For point~(iii), we prove the two implications separately. On the one hand, if $\Omega$ is bounded, then $w_\nu(\Omega)<+\infty$ for all directions, and therefore $w^{(N)}$ must be finite. On the other hand, if $w^{(N)}<+\infty$, then by point~(ii) it follows that $\Omega$ is contained in a bounded $N$-dimensional parallepiped.
\end{proof}

We now wish to prove that there exist pairs of boundary points suitably related to the principal widths. Before stating and proving this property, we need the following preliminary lemma. It is an adaptation of~\cite[Lem.~4.1]{FLS-of}, where the statement is proved only for $i=1$.
\begin{lemma}
\label{lem:Hausdorff-p-widths}
Let $(\Omega_n)_n\subset\BK_N$ be a sequence convergent to some $\Omega\in\BK_N$ in the Hausdorff metric. Then, the $i$-th principal width is continuous, that is,
\[
w^{(i)}(\Omega_n) \to w^{(i)}(\Omega).
\]
Additionally, if $(\nu_n^{(i)})_n \subset \mathbb{S}^{N-1}$ is a sequence of $i$-th principal directions of $(\Omega_n)_n$, any of its limit points $\nu^{(i)}$ is an $i$-th principal direction of $\Omega$.
\end{lemma}

\begin{proof}
Let us start by recalling some known facts from convex geometry. First, the Hausdorff convergence $\Omega_n\to \Omega$ is equivalent to the uniform convergence of the support functions, that is,
\begin{equation}
\label{eq:uniform_support_function}
\|\sigma_{\Omega_n}-\sigma_\Omega\|_\infty \to 0,
\end{equation}
refer to~\cite[Lem.~1.8.14]{Sch14}. Second, Hausdorff convergence implies that the corresponding sequence of diameters is uniformly bounded, that is, there exists $C>0$ such that
\begin{equation}
\label{eq:diameter_bound}
\diam(\Omega_n) \le C, \qquad \forall n.
\end{equation}
Third, fixed any convex and compact set $E$, the function $\nu\mapsto w_\nu(E)$ is Lip\-schitz, with Lip\-schitz constant given by the diameter of the set, that is, for all directions $\nu,u\in\mathbb{S}^{N-1}$, it holds that
\begin{equation}
\label{eq:directional-w_Lipschitz}
|w_\nu(E)-w_u(E)| \le \diam(E)\|\nu-u\|,
\end{equation}
refer to~\cite[Cor.~1.8.13]{Sch14}.

We can now turn our attention to proving the lemma. Let $(\Omega_n)_n$ be fixed and let $\Omega$ be its Hausdorff limit. For each $i=1,\dots, N$ and each $n$, by \cref{rem:well-posedness} we find unit vectors $\nu_n^{(i)}$ such that
\begin{equation}
\label{eq:choice_w-i-Omega_n}
\nu_n^{(i)} = \nu^{(i)}(\Omega_n),
\end{equation}
and such that
\begin{equation}
\label{eq:orthogonality-w_i^n}
\langle \nu_n^{(i)}, \nu_n^{(j)} \rangle = 0,\qquad\text{if } i\neq j.
\end{equation}

Let us consider the $N$ sequences $(\nu_n^{(i)})_n \subset\mathbb{S}^{N-1}$. As $\mathbb{S}^{N-1}$ is compact, up to subsequences, there exist $N$ unit vectors $\nu^{(i)}_{*}$ such that
\begin{equation}
\label{eq:convergenza_direzioni}
\nu_n^{(i)} \to \nu^{(i)}_{*}, \qquad \forall i=1,\dots,N.
\end{equation}
Further, the orthogonality~\eqref{eq:orthogonality-w_i^n} is preserved in the limit, that is,
\begin{equation}
\label{eq:limit_orthogonality}
\qquad \langle\nu^{(i)}_{*}, \nu^{(j)}_{*}\rangle = 0, \qquad \text{if } i\neq j.
\end{equation}

{\bf Step~(i).} We claim that
\begin{equation}
\label{eq:convergence-to-width-Omega}
w^{(i)}(\Omega_n) \to w_{\nu^{(i)}_{*}}(\Omega),\qquad \forall i =1,\dots,N.
\end{equation}
Indeed, by~\eqref{eq:choice_w-i-Omega_n}, by the Lipschitz continuity~\eqref{eq:directional-w_Lipschitz}, using the definition of directional width and the triangular inequality, we obtain
\begin{align*}
|w^{(i)}(\Omega_n)-w_{\nu^{(i)}_{*}}(\Omega)| 
&
=
|w_{\nu_n^{(i)}}(\Omega_n)-w_{\nu^{(i)}_{*}}(\Omega)|
\\
&
\le
|w_{\nu_n^{(i)}}(\Omega_n)-w_{\nu^{(i)}_{*}}(\Omega_n)|
+
|w_{\nu^{(i)}_{*}}(\Omega_n)-w_{\nu^{(i)}_{*}}(\Omega)|
\\
&
\le
\diam(\Omega_n)\|\nu_n^{(i)}-\nu^{(i)}_{*}\|
+
2\|\sigma_{\Omega_n} - \sigma_{\Omega}\|_\infty.
\end{align*}
By the equiboundedness of the diameters~\eqref{eq:diameter_bound}, the convergence of the unit vectors~\eqref{eq:convergenza_direzioni}, and the uniform convergence~\eqref{eq:uniform_support_function}, the RHS goes to $0$ when $n\to+\infty$, thus proving the claim.

{\bf Step~(ii).} We now show by induction that
\[
w^{(i)}(\Omega)=w_{\nu^{(i)}_{*}}(\Omega), \qquad \forall i=1,\dots,N.
\]

Let $i=1$. By definition of minimal width, we trivially have the inequality
\[
w_{\nu^{(1)}_{*}} (\Omega) \ge w^{(1)}(\Omega).
\]
To show the opposite inequality, let $\bar\nu^{(1)}$ be a direction realizing the minimal width of $\Omega$. Then, owing to the uniform convergence~\eqref{eq:uniform_support_function}, the definition of minimal width, and to~\eqref{eq:convergence-to-width-Omega}, we have
\[
w^{(1)}(\Omega) 
= 
w_{\bar\nu^{(1)}}(\Omega) 
= 
\lim_n w_{\bar \nu^{(1)}}(\Omega_n)
\ge
\lim_n w^{(1)}(\Omega_n)
=
w_{\nu^{(1)}_{*}}(\Omega).
\]
This also proves that any limit of $(\nu_n^{(1)})_n$ is a direction realizing the minimal width of the limit set $\Omega$.

Let now $i>1$, and assume that the claim holds for all indexes up to $i-1$. By virtue of the inductive assumption, for all $j<i$, we have that
\[
w^{(j)}(\Omega) = w_{\nu^{(j)}_{*}}(\Omega),
\]
and $\nu^{(j)}_{*}$ is a direction that realizes $w^{(j)}(\Omega)$.
By~\eqref{eq:limit_orthogonality}, $\nu^{(i)}_{*}$ is orthogonal to all these directions, thus, by definition  of $i$-th principal width,
\[
w_{\nu^{(i)}_{*}}(\Omega) \ge w^{(i)}(\Omega).
\]

Let us now prove the opposite inequality. Let $\bar \nu^{(i)}$ be a direction realizing the $i$-th principal width of $\Omega$. In particular, this means that $\langle\bar \nu^{(i)}, \nu^{(j)}_{*}\rangle = 0$ for all $j<i$.
For each $n$, we consider the partial Gram--Schmidt decomposition of $\bar \nu^{(i)}$ on $\vspan\{\nu_n^{(j)}, j<i\}$, that is, we let
\begin{equation*}
\bar\nu_n^{(i)}
=
\frac{\bar\nu^{(i)} - \displaystyle\sum_{j=1}^{i-1} \langle\bar\nu^{(i)}, \nu_n^{(j)}\rangle \nu_n^{(j)}}{\Big\|\bar\nu^{(i)} - \displaystyle\sum_{j=1}^{i-1} \langle\bar\nu^{(i)}, \nu_n^{(j)}\rangle \nu_n^{(j)} \Big\|}.
\end{equation*}

First, notice that for $n\gg 1$, this is well-defined, since $\langle\bar\nu^{(i)}, \nu^{(j)}_{*}\rangle = 0$ for all $j<i$, and $\nu_n^{(j)}$ converges to $\nu^{(j)}_{*}$. Second, $\bar\nu_n^{(i)} \in \mathbb{S}^{N-1}$ and it converges to $\bar \nu^{(i)}$. Third, $\bar \nu_n^{(i)}$ is orthogonal to $\nu_n^{(j)}$ for all $j<i$, thus it is a viable competitor in the definition of $w^{(i)}(\Omega_n)$.

Then, just as we did before, exploiting the uniform convergence~\eqref{eq:uniform_support_function}, the Lipschitz continuity~\eqref{eq:directional-w_Lipschitz}, the convergence~\eqref{eq:convergence-to-width-Omega} shown in Step~(i), the equiboundedness of the diameters~\eqref{eq:diameter_bound}, and the convergence of $\bar\nu_n^{(i)}$ just discussed, we get
\begin{align*}
w^{(i)}(\Omega)
&
=
w_{\bar \nu^{(i)}}(\Omega)
=
\lim_n w_{\bar \nu^{(i)}}(\Omega_n)
\\
&
=
\lim_n w_{\bar \nu_n^{(i)}}(\Omega_n)
+
\lim_n \big(w_{\bar \nu^{(i)}}(\Omega_n) - w_{\bar \nu_n^{(i)}}(\Omega_n)\big)
\\
&
\ge
\lim_n w^{(i)}(\Omega_n) - \lim_n \diam(\Omega_n)\|\bar\nu^{(i)} - \bar\nu_n^{(i)}\|
=
w_{\nu^{(i)}_{*}}(\Omega).
\end{align*}
This again proves both the equality and the fact that any limit point of the sequence $(\nu_n^{(i)})_n$ is a direction that realizes $w^{(i)}(\Omega)$.
Finally, since the limit points $\nu^{(1)}_*,\dots,\nu^{(N)}_*$ are pairwise orthonormal and they realize the principal widths $w^{(1)},\dots,w^{(N)}$, they form a set of principal directions of $\Omega$.
\end{proof}

\begin{proposition}
\label{prop:realization_principal_width}
Let $\Omega\in\BK_N$. For all $i=1,\dots,N$, there exist two points, $p^{(i)}\in F(\Omega,H_{\nu^{(i)}}(\Omega))$ and $q^{(i)}\in F(\Omega,H_{-\nu^{(i)}}(\Omega))$, such that
%
\[
p^{(i)}-q^{(i)} \in \vspan\{\nu^{(j)}\}_{j=1}^i,
\qquad
\text{and}
\qquad
\langle p^{(i)}-q^{(i)},\nu^{(i)}\rangle = w^{(i)},
\]
%
where $(\nu^{(1)},\dots,\nu^{(N)})$ is a set of principal directions of $\Omega$.
\end{proposition}

\begin{proof}
Since $\interior(\Omega)\ne\emptyset$, clearly $w^{(1)}(\Omega)>0$, and hence by \cref{prop:properties_p-widths}~(i) we have $w^{(i)}(\Omega)>0$ for all $i=1,\dots,N$.

\textbf{Step~(i).} We start by assuming that $\Omega$ is strictly convex. Strict convexity is equivalent to the differentiability everywhere in $\mathbb{R}^N\setminus\{0\}$ of (the $1$-homogeneous positive extension of) $\sigma_\Omega$, see, \emph{e.g.},~\cite[Thm.~1.7.4]{Sch14}.

Let $\nu^{(1)},\dots,\nu^{(i-1)}$ be the first $i-1$ vectors in a fixed set of principal directions of $\Omega$.
Then, the $i$-th principal width $w^{(i)}(\Omega)$ is the minimum of the constrained problem
\[
\inf\{\, 
w_{\nu}(\Omega)\,:\,  \|u\|^2=1,\, \langle u, \nu^{(j)}\rangle = 0,\, \forall j=1,\dots\,i-1 \, \},
\]
which we know to be realized by an $i$-th principal direction $\nu^{(i)}$.
The $i$ constraints
\[
g_j(u) = \langle u, \nu^{(j)}\rangle, \, \forall j=1,\dots,i-1 \qquad \text{and} \qquad g_i(u)=\|u\|^2-1
\]
are of class $\mathcal{C}^1$. We claim that the Jacobian matrix of $g = (g_j)_{j=1}^i$, that is,
\[
Jg(u) = (\nabla g_j)_{j=1}^i = 
\left(
    \nu^{(1)},
    \quad
    \cdots 
    \quad
    \nu^{(i-1)}, 
    \quad
    2u
\right)^T
\]
has maximal rank on the set $\{\,\|u\|^2=1,\, \langle u, \nu^{(j)}\rangle = 0,\, \forall j=1,\dots,i-1\,\}$. Clearly, the first $i-1$ rows are pairwise orthogonal, and thus linearly independent. If the last row were a linear combination of the previous ones, we would have
\[
2u = \displaystyle \sum_{j=1}^{i-1} \alpha_j \nu^{(j)}
\]
for some $\alpha_1,\dots,\alpha_{i-1}\in\R$ that are not all $0$. If we take $j$ such that $\alpha_j\ne0$, then $\langle u, \nu^{(j)}\rangle  = \alpha_j/2 \ne0$, against the prescribed constraint.

Since the function $f(u) = \sigma_\Omega(u)+\sigma_\Omega(-u)$ is differentiable in a neighborhood of $\mathbb{S}^{N-1}$, being the sum of two differentiable functions in $\R^N\setminus\{0\}$, we can apply\footnote{This theorem is usually stated assuming $f$ of class $\mathcal{C}^1$, but it is actually sufficient for it to be differentiable, since its regularity comes into play only to use the Chain Rule. The $\mathcal{C}^1$ regularity of $g$, on the other hand, cannot be weakened, as it is needed to apply the Implicit Function Theorem.} Lagrange Multipliers Theorem: there exist $\mu_1,\dots,\mu_i\in\R$ such that
\begin{align*}
\nabla f(\nu^{(i)}) 
&
=
\nabla\sigma_\Omega(\nu^{(i)})-\nabla\sigma_\Omega(-\nu^{(i)}) 
\\
&
=
\displaystyle\sum_{j=1}^i \mu_j \nabla g_j
=
\displaystyle\sum_{j=1}^{i-1} \mu_j \nu^{(j)} + 2\mu_i \nu^{(i)}.
\end{align*}

Since $\Omega$ is strictly convex, one has
\[
\{\nabla\sigma_\Omega(u)\} = F(\Omega,H_u(\Omega)), \qquad \forall u\in\mathbb{S}^{N-1},
\]
refer to~\cite[Cor.~1.7.3 and Thm.~1.7.4]{Sch14}. Therefore, from the two equalities above we have
\[
p^{(i)} - q^{(i)} \in \vspan\{\nu^{(j)}\}_{j=1}^i,
\]
being $p^{(i)}$, resp., $q^{(i)}$, the unique point in $F(\Omega,H_{\nu^{(i)}}(\Omega))$, resp., $F(\Omega,H_{-\nu^{(i)}}(\Omega))$. Further, 
\[
\langle p^{(i)}-q^{(i)}, \nu^{(i)} \rangle
= 
\dist(H_{\nu^{(i)}}(\Omega); H_{-\nu^{(i)}}(\Omega)) 
= 
\sigma_\Omega(\nu^{(i)}) + \sigma_\Omega(-\nu^{(i)}) 
= 
w^{(i)}(\Omega).
\]

\textbf{Step~(ii).} We can now conclude via an approximation argument. Indeed, any set $\Omega\in\BK_N$ can be approximated in the Hausdorff metric by a sequence of bounded and strictly convex sets $(\Omega_n)_n$, see \cite[Thm.~3.4.1(c)]{Sch14}. We recall again that this is equivalent to the uniform convergence of the support functions, see~\cite[Lem.~1.8.14]{Sch14}.

Let now $(\nu_n^{(1)},\dots,\nu_n^{(N)})$ be a set of principal directions of $\Omega_n$, and let $p_n^{(i)}$, resp., $q_n^{(i)}$, the unique point in $F(\Omega_n, H_{\nu_n^{(i)}}(\Omega_n))$, resp., $F(\Omega_n, H_{-\nu_n^{(i)}}(\Omega_n))$, that we found in Step~(i). In particular, for every $i=1,\dots,N$ we have
\begin{equation}
\label{eq:realization-pair}
\begin{split}
(i)\quad &p_n^{(i)} - q_n^{(i)} \in \vspan\{\nu^{(j)}_n\}_{j=1}^i,
\\
(ii) \quad &\langle p_n^{(i)}-q_n^{(i)}, \nu_n^{(i)} \rangle
=
w^{(i)}(\Omega_n).
\end{split}
\end{equation}

By \cref{lem:Hausdorff-p-widths}, the $N$-tuples $(\nu_n^{(1)},\dots,\nu_n^{(N)})_n$ converge up to subsequences to a set of principal directions $(\nu^{(1)},\dots,\nu^{(N)})$ of $\Omega$.
On the other hand, by compactness of $\Omega$ and the Hausdorff convergence of $\Omega_n$ to $\Omega$, we can assume the convergence up to subsequences of $(p_n^{(i)})_n$, resp., of $(q_n^{(i)})_n$, to some $p^{(i)}$, resp., to some $q^{(i)}$, in $\partial \Omega$. 

The two properties in~\eqref{eq:realization-pair} pass to the limit, the second one also owing to \cref{lem:Hausdorff-p-widths}. Thus, in order to conclude, it only remains to show that
\[
p^{(i)}\in F(\Omega, H_{\nu^{(i)}}(\Omega))\qquad \text{and} \qquad q^{(i)}\in F(\Omega, H_{-\nu^{(i)}}(\Omega)).
\]
First, notice that we have the convergence of $\sigma_{\Omega_n}(\nu_n^{(i)})$ to $\sigma_\Omega(\nu^{(i)})$. Indeed, fixed $\varepsilon>0$, we find $\bar n$ such that for all $n>\bar n$, it holds that
\begin{align*}
|\sigma_{\Omega_n}(u)-\sigma_\Omega(u)| 
&
\le 
\varepsilon,
\qquad
\forall u\in\mathbb{S}^{N-1}
\\
|\sigma_{\Omega}(\nu_n^{(i)})-\sigma_\Omega(\nu^{(i)})| 
&
\le 
\varepsilon,
\end{align*}
the former by virtue of the uniform convergence of $(\sigma_{\Omega_n})_n$, the latter thanks to the convergence of $(\nu_n^{(i)})_n$ to $\nu^{(i)}$ and the continuity of $\sigma_\Omega$. Then, passing to the limit in
\[
\langle p_n^{(i)}, \nu_n^{(i)}\rangle = \sigma_{\Omega_n}(\nu_n^{(i)}),
\]
which is the equality defining the supporting hyperplane of $\Omega_n$ with outer direction $\nu_n^{(i)}$, see~\eqref{eq:supporting_hyperplane}, we get
\[
\langle p^{(i)}, \nu^{(i)}\rangle = \sigma_\Omega(\nu^{(i)}),
\]
which shows that $p^{(i)}$ belongs to $F(\Omega, H_{\nu^{(i)}}(\Omega))$, as claimed. Analogously, one also gets $q^{(i)}\in F(\Omega,H_{-\nu^{(i)}}(\Omega))$, which concludes the proof.
\end{proof}

\subsection{Convex sets with small relative width}
\label{sec:relative_width}

Now, we can finally introduce the notion of \emph{relative width} of a set, which is crucial to the sufficient condition of maximizing sequences that we will prove in the next section.

\begin{definition}[Relative width]
\label{def:relative_width}
Let $\Omega\in\BK_N$ with $N\ge2$. We call \textit{relative width} of $\Omega$ the ratio 
\[
w^{(\mathrm{rel})}(\Omega) = \frac{w^{(1)}(\Omega)}{w^{(2)}(\Omega)},
\]
where the widths $w^{(1)}$ and $w^{(2)}$ are defined as in \cref{def:p-widths}.
We say that a sequence $(\Omega_n)_{n}\subset\BK_N$ has \emph{vanishing relative width} if
\(
w^{(\mathrm{rel})}(\Omega_n) \to 0.
\)
\end{definition}

\begin{theorem}
\label{thm:smallrelwidth}
Let $\Omega\in\BK_3$, with minimal width $w^{(1)}= w^{(1)}(\Omega)$ and relative width $w^{(\mathrm{rel})}= w^{(\mathrm{rel})}(\Omega) < 1/100$. Then, $\Omega$ contains an isometric copy of the parallelepiped 
\[
\left[0,w^{(1)}(1-10\sqrt{w^{(\mathrm{rel})}})\right]\times \left[0,w^{(1)}/(50\sqrt{w^{(\mathrm{rel})}})\right]^{2}.
\]
\end{theorem}

The proof of \cref{thm:smallrelwidth} requires some preliminary results. 

\begin{lemma}
\label{lem:3p-inradius}
Let $T$ be a triangle whose angles are all greater than or equal to $\alpha >0$ and whose smallest side has length $L>0$. Then, its inradius $\rho=\rho(T)$ satisfies
\[
\rho \ge \frac{L}{2}\tan\left(\frac{\alpha}{2}\right).
\]
\end{lemma}
\begin{proof}
    Let $A,B$ be two vertexes of $T$ and let $O$ be the center of the inscribed circle of $T$. Take $P$ to be the projection of $O$ onto $AB$ and assume without loss of generality that $\overline{AP}\ge \overline{PB}$. Since the triangle $OAP$ is rectangle in $P$, and the measure of $\widehat{OAP}$ is greater or equal to $\alpha/2$, we deduce that 
    \[
    \rho = \overline{OP} = \overline{AP} \tan(\widehat{OAP}) \ge \frac{L}{2}\tan\left(\frac{\alpha}{2}\right)\,.\qedhere
    \]
\end{proof}

\begin{corollary}
\label{cor:3p-inradius}
Let $T$ be a triangle satisfying the assumptions of \cref{lem:3p-inradius} and let $h>0$. Then, the prism $P = [0,h]\times T$ contains a translated copy of the parallelepiped $[0,h]\times [0,\ell]^{2}$, where 
\[
\ell =  \sqrt{2}\rho(T) = \frac{L}{\sqrt 2}\tan\left(\frac{\alpha}{2}\right).
\]
\end{corollary}

Now, we turn to the proof of \cref{thm:smallrelwidth}.

\begin{proof}[Proof of \cref{thm:smallrelwidth}]
For each $i=1,2,3$, we set $w^{(i)}\in [0,+\infty)$, $\nu^{(i)}\in \mathbb{S}^2$, and $(p^{(i)},q^{(i)})\in \Omega\times \Omega$ according to \cref{def:p-widths} and \cref{prop:realization_principal_width}. In particular, we have
\begin{align*}
&p^{(1)}-q^{(1)} = w^{(1)}\nu^{(1)},\\
&\langle p^{(2)}-q^{(2)}, \nu^{(2)}\rangle = w^{(2)},\quad \text{and}\quad \langle p^{(2)}-q^{(2)}, \nu^{(3)}\rangle = 0,\\
&\langle p^{(3)}-q^{(3)}, \nu^{(3)}\rangle = w^{(3)}.
\end{align*}
For the sake of a simpler notation, given $z = (z_{1},z_{2},z_{3})\in \R^{3}$ we shall write $z_\perp = (z_{2},z_{3})$, and we set $w=w^{(2)}=(w^{(\mathrm{rel})})^{-1}>100$.

Up to rescaling we can assume $w^{(1)} = 1$, and up to a rigid motion we can assume that $\nu^{(i)} = e_{i}$ for all $i=1,\dots,N$ and that $q^{(1)} = (0,0,0)$ and $p^{(1)}=(1,0,0)$.

Then, we set $S_{i} = \{q^{(i)},p^{(i)}\}$ and we define
\[
S = \begin{cases}
S_{1}\cup S_{2} & \text{if }|q^{(2)}_3| \ge\displaystyle \frac{w^{(2)}}{8}\\
S_{1}\cup \{a,b\} & \text{otherwise,}
\end{cases}
\]
where $a\in S_{2}$ and $b\in S_{3}$ are such that 
\begin{equation}\label{eq:Scaso2}
|a_{2}| \ge \frac{w^{(2)}}{2}\quad \text{and}\quad
|b_{3}| \ge \frac{w^{(3)}}{2}. 
\end{equation}
Note that $a$ and $b$ satisfying the above requirements necessarily exist by the definition of $S_{i}$ for $i=1,2,3$. 

The proof proceeds by showing that the convex hull of $S$, denoted by $\coh(S)$ from now on, contains an isometric copy of the parallelepiped defined in the statement. 

\medskip

\textbf{Case 1: $|q^{(2)}_3| \ge w/8$.} Here, we assume without loss of generality that $p^{(2)}_3$ and $q^{(2)}_3$ are both positive, and thus equal to $|q^{(2)}_3|$, see \cref{img:Case_1}.
\begin{figure}[t]
\centering
\begin{tikzpicture}[scale=0.4]
  \coordinate (E) at (0.00, 0.00);
  \coordinate (F) at (-5.00, 5.00);
  \coordinate (G) at (3.00, 5.00);
  \coordinate (H) at (-1.77, 1.77);
  \coordinate (I) at (1.06, 1.77);
  \coordinate (J) at (-0.35, 3.42);
  \coordinate (K) at (-8.00, 0.00);
  \coordinate (L) at (6.00, 0.00);
  \coordinate (M) at (0.00, 8.00);
  \coordinate (N) at (0.00, -4.00);

  \draw[thick, densely dashed] (-5,-3) rectangle (3,7);
  \fill[gray!20] (F) -- (E) -- (G) -- cycle;
  \draw[thick] (F) -- (E) -- (G) -- cycle;
  \fill[gray!60] (H) -- (I) -- (J) -- cycle;
  \draw[thick] (H) -- (I) -- (J) -- cycle;
  \draw[ultra thin, ->] (K) -- (L);
  \draw[ultra thin, ->] (N) -- (M);
  \fill[black] (E) circle (3pt);
  \node[below left, black] at (E) {$0$};
  \fill[black] (F) circle (3pt);
  \node[left, black] at (F) {$q^{(2)}_{\perp}$};
  \fill[black] (G) circle (3pt);
  \node[right, black] at (G) {$p^{(2)}_{\perp}$};
  \fill[black] (H) circle (3pt);
  \node[below left, black] at (H) {$\sigma$};
  \fill[black] (I) circle (3pt);
  \node[below right, black] at (I) {$\tau$};
  \fill[black] (J) circle (3pt);
  \node[above, black] at (J) {$\xi$};
\end{tikzpicture}
\caption{The geometric situation of Case~1.}
\label{img:Case_1}
\end{figure}
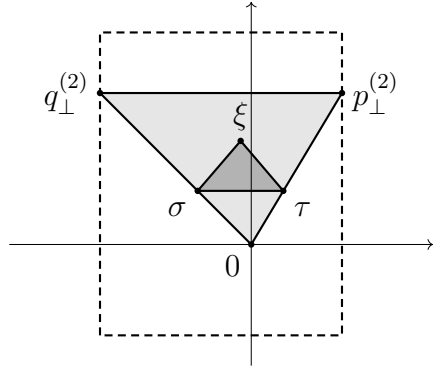
Let 
\[
\sigma = \frac{1}{\sqrt{w}} q^{(2)}_{\perp},\qquad 
\tau = \frac{1}{\sqrt{w}} p^{(2)}_{\perp},\qquad
\xi =\frac{1}{2}(\sigma + \tau) + \left(0,\frac{\sqrt{w}}{2}\right).
\] 
The triangle $T$ of vertexes $\sigma,\tau,\xi$ is right-angled at $\xi$ and isosceles by construction. Moreover, $T$ is contained in the triangle $Q$ obtained as the projection of $\coh(S)$ onto the $e_{1}^{\perp}$ coordinate plane (note that the vertexes of $Q$ are $p^{(2)}_{\perp}$, $q^{(2)}_{\perp}$, and the origin). Indeed, it is immediate to check that $\sigma$ and $\tau$ are contained in $Q$ and satisfy
\[
\|\sigma - \tau\| = \frac{1}{\sqrt{w}}\,\|p^{(2)}_2-q^{(2)}_2\| = \sqrt{w}.
\]
It remains to show that $\xi$ belongs to $Q$. To this aim we only have to check that its third coordinate does not exceed $|q^{(2)}_3|$, which amounts to require that
\begin{equation}\label{eq:condd2}
\xi_{3} = \frac{\sqrt{w}}{2} + \frac{|q^{(2)}_3|}{\sqrt{w}} \le |q^{(2)}_3|.
\end{equation}
Condition \eqref{eq:condd2} is equivalent to
\[
|q^{(2)}_3| \ge \frac{w}{2(\sqrt{w}-1)}
\]
which is satisfied when
\[
\frac{1}{8} \ge \frac{1}{2(\sqrt{w}-1)}\quad \Longleftrightarrow\quad w \ge 25,
\]
thanks to the standing assumption $|q^{(2)}_3|\ge w/8$.
Hence, since we assumed $w=(w^{(\mathrm{rel})})^{-1}>100$, we can conclude that $T\subset Q$.

By \cref{lem:3p-inradius} applied to the triangle $T$ we obtain  
\[
\rho(T) \ge \frac{\sqrt{w}}{4}\tan(\pi/8) \ge \frac{\sqrt{w}}{10}.
\]
Let now $[\sigma_{1}^{-},\sigma_{1}^{+}]$ be the interval such that
\[
(t,\sigma) \in \coh(S) \quad \Longleftrightarrow\quad t\in [\sigma_{1}^{-},\sigma_{1}^{+}],
\]
and similarly define the intervals $[\tau_{1}^{-},\tau_{1}^{+}]$ and $[\xi_{1}^{-},\xi_{1}^{+}]$. We now estimate $\sigma_{1}^{\pm},\tau_{1}^{\pm},\xi_{1}^{\pm}$. By convexity, we have
\[
\sigma_{1}^{-}\le \frac{q^{(2)}_1}{\sqrt{w}} \le \frac{1}{\sqrt{w}} \qquad \text{and}\qquad  \sigma_{1}^{+}\ge 1- \frac{1}{\sqrt{w}} + \frac{q^{(2)}_1}{\sqrt{w}} \ge 1-\frac{1}{\sqrt{w}}.
\]
Similarly,
\[
\tau_{1}^{-}\le \frac{p^{(2)}_1}{\sqrt{w}} \le \frac{1}{\sqrt{w}} \qquad \text{and}\qquad  \tau_{1}^{+}\ge 1- \frac{1}{\sqrt{w}} + \frac{p^{(2)}_1}{\sqrt{w}} \ge 1-\frac{1}{\sqrt{w}}.
\]
Then, in order to estimate $\xi_{1}^{\pm}$, we notice that 
\[
\frac{\xi_{3}}{|q^{(2)}_3|} = \frac{\sqrt{w}}{2|q^{(2)}_3|} + \frac{1}{\sqrt{w}} \le \frac{5}{\sqrt{w}},
\]
and hence arguing as before we get
\[
\xi_{1}^{-} \le \frac{5}{\sqrt{w}}\qquad \text{and}\qquad 
\xi_{1}^{+} \ge 1- \frac{5}{\sqrt{w}}.
\]
Consequently, the interval $[\frac{5}{\sqrt{w}},1 - \frac{5}{\sqrt{w}}]$ is contained in the intersection of the previous three intervals. By convexity, the prism $P = [\frac{5}{\sqrt{w}},1 - \frac{5}{\sqrt{w}}] \times T$ is contained in $\coh(S)$. Finally, \cref{cor:3p-inradius} allows us to conclude that, in this case, $\coh(S)$ contains a translated copy of the parallelepiped 
\[
\left[0,1-\frac{10}{\sqrt{w}}\right]\times \left[0,\frac{\sqrt 2}{10}\sqrt{w}\right]^{2}.
\]
\medskip

\textbf{Case 2: $|q^{(2)}_3|< w/8$.} We start by estimating the angle $\alpha\in (0,\pi)$ defined by the two vectors $a_{\perp}$ and $b_{\perp}$. Thanks to \eqref{eq:Scaso2}, up to a symmetry, we can assume $a_{2}\le -w/2$ and $b_{3}\le -w/2$.

It is easy to check that $\alpha$ contains the angle $\alpha_{\min}$ defined by the vectors $(a_2,-w/8)$ and $(a_2,-w/2)$. Since $-w\le a_2\le-w/2$, there exists $x \in [0,1]$ such that $a_2 = -w(1+x)/2$, see \cref{img:estimate_alpha}.
\begin{figure}[t]
    \centering
    \begin{subfigure}[b]{0.45\textwidth}
    \centering
    \begin{tikzpicture}[scale=0.4]
    
    \coordinate (A) at (0.00, 0.00);
    \coordinate (B) at (-5.00, -1.00);
    \coordinate (C) at (-5.00, -4.00);
    \coordinate (D) at (-5.00, 0.00);
    \coordinate (E) at (0.00, -4.00);
    \coordinate (F) at (-2.80, 0.00);
    \coordinate (K) at (-8.00, 0.00);
    \coordinate (L) at (6.00, 0.00);
    \coordinate (M) at (0.00, -8.00);
    \coordinate (N) at (0.00, 4.00);
    
    \draw[thick, densely dashed] (-5,-7) rectangle (3,3);
    \draw[ultra thin, ->] (K) -- (L);
    \draw[ultra thin, ->] (M) -- (N);
    \draw pic [draw, blue, fill=blue, fill opacity=0.2, angle radius=12pt] {angle = B--A--C};
    \draw pic [draw, red, fill=red, fill opacity=0.2, angle radius=8pt] {angle = D--A--B};
    \draw pic [draw, green, fill=green, fill opacity=0.2, angle radius=8pt] {angle = C--A--E};
    \draw[decorate, decoration={brace, amplitude=7pt}] (-5,3.2) -- (0,3.2) node[midway, above=5pt] {$\frac{1+x}{2}w$};
    \draw[decorate, decoration={brace, amplitude=7pt}] (3,-7.2) --
    (-5,-7.2) node[midway, below=5pt] {$w$};
    \draw[thick, blue] (A) -- (B);
    \draw[thick, blue] (A) -- (C);
    \fill[black] (A) circle (4pt);
    \fill[black] (B) circle (4pt);
    \fill[black] (C) circle (4pt);
    \node[below, black] at (B) {$(a_2,-w/8)$};
    \node[below, black] at (C) {$(p^{(2)}_2,-w/2)$};
    \node[above right, black] at (A) {$0$};
    \node[above right, blue] at (F) {$\alpha_{\min}$};
    \end{tikzpicture}
        \caption{Estimate of $\alpha_{\min}$.}
        \label{subfig:alpha_min}
    \end{subfigure}
    \begin{subfigure}[b]{0.45\textwidth}
        \centering
        \begin{tikzpicture}[scale=0.4]
    
    \coordinate (A) at (0.00, 0.00);
    \coordinate (B) at (-5.00, 1.00);
    \coordinate (C) at (3.00, -4.00);
    \coordinate (D) at (-5.00, 0.00);
    \coordinate (E) at (0.00, -4.00);
    \coordinate (F) at (-2.00, 0.00);
    \coordinate (K) at (-8.00, 0.00);
    \coordinate (L) at (6.00, 0.00);
    \coordinate (M) at (0.00, -8.00);
    \coordinate (N) at (0.00, 4.00);
    
    \draw[thick, densely dashed] (-5,-7) rectangle (3,3);
    \draw[ultra thin, ->] (K) -- (L);
    \draw[ultra thin, ->] (M) -- (N);
    \draw pic [draw, blue, fill=blue, fill opacity=0.2, angle radius=8pt] {angle = B--A--C};
    \draw pic [draw, red, fill=red, fill opacity=0.2, angle radius=12pt] {angle = B--A--D};
    \draw pic [draw, green, fill=green, fill opacity=0.2, angle radius=12pt] {angle = E--A--C};
    \draw[decorate, decoration={brace, amplitude=7pt}] (-5,3.2) -- (0,3.2) node[midway, above=5pt] {$\frac{1+x}{2}w$};
    \draw[decorate, decoration={brace, amplitude=7pt}] (3,-7.2) --
    (-5,-7.2) node[midway, below=5pt] {$w$};
    \draw[thick, blue] (A) -- (B);
    \draw[thick, blue] (A) -- (C);
    \fill[black] (A) circle (4pt);
    \fill[black] (B) circle (4pt);
    \fill[black] (C) circle (4pt);
    \node[left, black] at (B) {$(a_2,w/8)$};
    \node[below, black] at (C) {$(p^{(2)}_2,-w/2)$};
    \node[above right, black] at (A) {$0$};
    \node[below, blue] at (F) {$\alpha_{\max}$};
    \end{tikzpicture}
    \caption{Estimate of $\alpha_{\max}$.}
    \label{subfig:alpha_max}
    \end{subfigure}
    \caption{Estimates of the two angles $\alpha_{\min}$ and $\alpha_{\max}$}
    \label{img:estimate_alpha}
\end{figure}
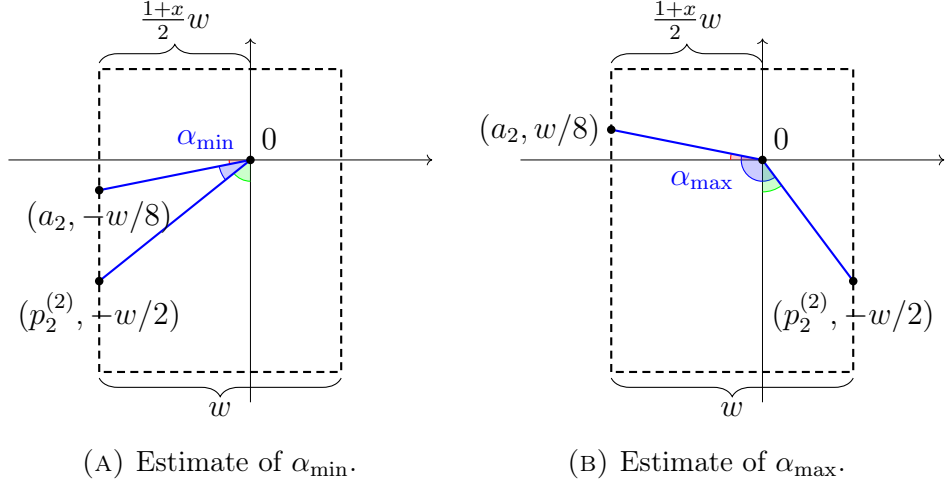

We have that
\[
\alpha_{\min} = \alpha_{\min}(x) = \frac{\pi}{2} - \arctan\left(\frac{1}{4(1+x)}\right) - \arctan(1+x),
\]
and we can observe that the function $x\mapsto\alpha_{\min}(x)$ attains its minimum at $x=1$. This shows that
\[
\alpha \ge \alpha_{\min}(1) = \frac{\pi}{2} - \arctan\left(\frac{1}{8}\right) - \arctan(2) > \frac{\pi}{12}.
\]
We can similarly estimate $\alpha$ from above. To this aim, notice that $\alpha$ is contained in the angle $\alpha_{\max}$ defined by the vectors $(a_2,w/8)$ and $(p^{(2)}_2,-w/2)$. Then, setting as before $x\in [0,1]$ such that $a_2 = -w(1+x)/2$, we obtain
\[
\alpha_{\max}(x) = \frac{\pi}{2} + \arctan\left(\frac{1}{4(1+x)}\right) + \arctan(1-x).
\]
Noting that $\alpha_{\max}(x)$ is maximized at $x=0$, we deduce 
\[
\alpha \le \alpha_{\max}(0) = \frac{\pi}{2} + \arctan\left(\frac{1}{4}\right) + \arctan(1) < \frac{5\pi}{6}.
\]
In conclusion, we have 
\[
\frac{\pi}{12}< \alpha < \frac{5\pi}{6}.
\]
Now, since we assumed $\|a_{\perp}\|, \|b_{\perp}\| \ge w/2$, we can consider the vectors $\tilde a$ and $\tilde b$ obtained by scaling $a_{\perp}$ and $b_{\perp}$ in such a way that 
\[
\|\tilde a\| = \|\tilde b\| = \sqrt{w} < \frac{w}{2},
\] 
and then define $T$ as the triangle with vertexes $\tilde a$, $\tilde b$, and the origin, see \cref{img:Case_2}.
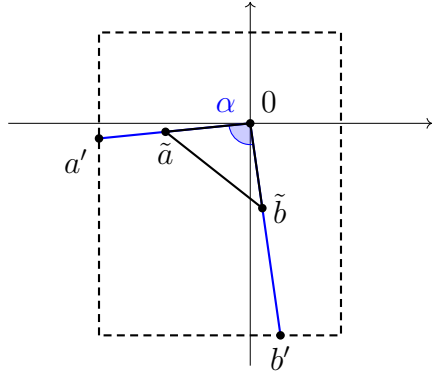
\begin{figure}[t]
\centering
\begin{tikzpicture}[scale=0.4]
    
    \coordinate (A) at (0.00, 0.00);
    \coordinate (B) at (-5.00, -0.50);
    \coordinate (C) at (1.00, -7.00);
    \coordinate (D) at (-2.80, -0.28);
    \coordinate (E) at (0.40, -2.80);
    \coordinate (F) at (-1.50, 0.00);
    \coordinate (K) at (-8.00, 0.00);
    \coordinate (L) at (6.00, 0.00);
    \coordinate (M) at (0.00, -8.00);
    \coordinate (N) at (0.00, 4.00);
    
    \draw[thick, densely dashed] (-5,-7) rectangle (3,3);
    \draw[ultra thin, ->] (K) -- (L);
    \draw[ultra thin, ->] (M) -- (N);
    \draw pic [draw, blue, fill=blue, fill opacity=0.2, angle radius=8pt] {angle = B--A--C};
    \draw[thick, blue] (A) -- (B);
    \draw[thick, blue] (A) -- (C);
    \fill[black] (A) circle (4pt);
    \fill[black] (B) circle (4pt);
    \fill[black] (C) circle (4pt);
    \fill[black] (D) circle (4pt);
    \fill[black] (E) circle (4pt);
    \draw[thick] (D) -- (E) -- (A) -- cycle;
    \node[below left, black] at (B) {$a'$};
    \node[below, black] at (C) {$b'$};
    \node[below, black] at (D) {$\tilde a$};
    \node[right, black] at (E) {$\tilde b$};
    \node[above right, black] at (A) {$0$};
    \node[above right, blue] at (F) {$\alpha$};
\end{tikzpicture}
\caption{The geometric situation of Case 2.}
\label{img:Case_2}
\end{figure}
It is then easy to check that $T$ is isosceles and non-degenerate, with smallest angle $\alpha>\pi/12$ and with smallest side length $L$ satisfying 
\[
L\ge 2\sin\left(\frac{\pi}{24}\right)\sqrt{w} > \frac{\sqrt{w}}{4}
\]
By similar calculations to those of Case 1, the intervals of first coordinates of points projecting onto $\tilde a$ and $\tilde b$ can be shown to contain the common interval
\[
\left[\frac{1}{2\sqrt{w}}, 1-\frac{1}{2\sqrt{w}}\right].
\]
Again, by \cref{cor:3p-inradius} we obtain that $\coh(S)$ contains an isometric copy of a parallelepiped of the form 
\[
\left[\frac{1}{2\sqrt{w}}, 1-\frac{1}{2\sqrt{w}}\right]\times \left[0,\frac{\sqrt{w}}{50}\right]^{2},
\]
which proves that $\coh(S)$ contains an isometric copy of the parallelepiped
\[
\left[0,1-\frac{1}{\sqrt{w}}\right] \times \left[0,\frac{\sqrt{w}}{50}\right]^{2}.
\qedhere
\]
\end{proof}

\section{Sufficient conditions for maximizing sequences}
\label{sec:maximizing_sequences}

We prove here our main result, which gives a sufficient condition for sequences in $\BK_3$ to be maximizing for the reverse Cheeger inequality.

\begin{theorem}
\label{thm:3-dim-reverse-cheeger}
    Every sequence $(\Omega_n)_n\subset\BK_3$ with vanishing relative width is maximizing for the reverse Cheeger inequality \eqref{eq:Buser}.
\end{theorem}
\begin{proof}
    For the sake of a simpler notation, for each $n$ we set $w^{(i)}_n=w^{(i)}(\Omega_n)$ with $i=1,2,3$ and $w^{(\mathrm{rel})}_n=w^{(\mathrm{rel})}(\Omega_n)$.
    In addition, we can rescale the set $\Omega_n$ in such a way that $w_n^{(1)}=1$, and thus $w^{(\mathrm{rel})}_n=1/w_n^{(2)}$.
    
    By \cref{prop:properties_p-widths} we know that each set $\Omega_n$ is contained in a parallelepiped $P_n$ with edges of length $w_n^{(1)},w_n^{(2)},w_n^{(3)}$. Hence, by the monotonicity of $\lambda_1$, we get
    \begin{equation}
    \label{eq:estimate_Laplacian}
        \lambda_1(\Omega_n)\ge\lambda_1(P_n)=\pi^2\sum_{i=1}^3\frac{1}{\big(w_n^{(i)}\big)^2}\ge\pi^2.
    \end{equation}
    
    On the other hand, since the sequence has vanishing relative width, we can assume without loss of generality that every $\Omega_n$ satisfies the assumptions of \cref{thm:smallrelwidth}, and thus it contains a copy of the parallelepiped
    \[
        R_n=\left[0,1-10\sqrt{1/w_n^{(2)}}\right]\times \left[0,\left(50\sqrt{1/w_n^{(2)}}\right)^{-1}\right]^{2}.
    \]
    Hence, by the monotonicity of the Cheeger constant, we get
    \begin{align}
    \label{eq:estimate_Cheeger}
        h(\Omega_n)&\le h(R_n)\le\frac{\Per(R_n)}{|R_n|} \nonumber \\
        &=\frac{2\left(50\sqrt{1/w_n^{(2)}}\right)^{-2}+4\left(1-10\sqrt{1/w_n^{(2)}}\right)\left(50\sqrt{1/w_n^{(2)}}\right)^{-1}}{\left(1-10\sqrt{1/w_n^{(2)}}\right) \left(50\sqrt{1/w_n^{(2)}}\right)^{-2}} \nonumber\\
        &=\frac{2}{1-10\sqrt{1/w_n^{(2)}}}+4\left(50\sqrt{1/w_n^{(2)}}\right).
    \end{align}
    
    Combining the two estimates \eqref{eq:estimate_Laplacian} and \eqref{eq:estimate_Cheeger} we obtain
    \[
        J(\Omega_n)\ge\frac{\pi^2}{\left[\frac{2}{1-10\sqrt{1/w_n^{(2)}}}+ 4\left(50\sqrt{1/w_n^{(2)}}\right)\right]^2}\to\frac{\pi^2}{4},
    \]
    thanks to the assumption that $w_n^{(2)}\to+\infty$.
    By \cref{prop:rev_Cheeger} we can conclude that $J(\Omega_n)\to\pi^2/4$, and thus $(\Omega_n)_n$ is a maximizing sequence for \eqref{eq:Buser}.
\end{proof}

\subsection{Rhomboid-like sets}
\label{sec:rhomboid}
The proof of \cref{thm:smallrelwidth}, and consequently \cref{thm:3-dim-reverse-cheeger}, seems difficult to extend to the case $N\ge4$, since it heavily relies on $3$-dimensional geometric constructions.
In this section, however, we will prove that a slightly stronger assumption on the principal widths of the sets $\Omega_n$ is sufficient to ensure the maximization of the reverse Cheeger inequality in any dimension.

\begin{definition}
\label{defin:rhomboid}
Let $N\ge2$ be fixed and let $\big(v^{(i)}_\pm\big)_i$ be a collection of $2N$ distinct points of the form
\[
v^{(i)}_\pm = \alpha^{(i)}_\pm e_i \qquad \forall i=1,\dots,N
\]
where $\alpha^{(i)}_+ \ge0$ and $\alpha^{(i)}_- \le0$ are not both equal to $0$ for each $i$.

The polyhedron $\Delta$ given by the convex hull of these $2N$ points will be called a \emph{rhomboid}; note that the vertexes of $\Delta$ are contained in $\big\{v^{(i)}_\pm\,:\, i=1,\dots,N\big\}$. We assume without loss of generality that $v^{(i)}_\pm$ is a vertex of $\Delta$ for all $i$. Then, we call \emph{diagonal} of $\Delta$ any segment $d^{(i)} = [v^{(i)}_-,v^{(i)}_+]$. In particular, we have
\[
\|v^{(i)}_\pm\| = |\alpha^{(i)}_\pm| \qquad\text{and}\qquad \|d^{(i)}\| = \|v^{(i)}_+\| + \|v^{(i)}_-\| = \alpha^{(i)}_+-\alpha^{(i)}_->0.
\]
\end{definition}

\begin{lemma}
\label{lem:rhomboid}
Let $\Delta$ be an $N$-dimensional rhomboid of vertexes $\{v^{(i)}_{\pm}\}_{i=1}^N$ and assume that $\|v^{(i)}_+\|\ge\|v^{(i)}_-\|$ for every $i=1,\dots,N$.
Then, $\Delta$ contains an isometric copy of the $N$-dimensional parallelepiped
\[
    \left[0,\left(1-\frac{1}{\rho}\right)\|d^{(1)}\|\right]\times\left[0,\rho\right]^{N-1}
\]
where 
\begin{equation}
\label{eq:rhombus-assumption}
\rho = \sqrt{\frac{1}{N-1}\, \min\{\, \|v^{(i)}_+\|,\,i=2,\dots,N\}}\,.    
\end{equation}
\end{lemma}

\begin{proof}
We start by proving this lemma in dimension $N=2$ and we refer the reader to \cref{fig:rhomboid} to properly follow the construction leading to the claim.

\textbf{The case $N=2$.} Let $\sigma = (\sigma_1, \sigma_2)$ be a point on $\partial \Delta$, lying on the segment from $v^{(1)}_+$ to $v^{(2)}_+$, such that
\[
\sigma_2 = \sqrt{\|v^{(2)}_+\|}.
\]
By Thales Theorem
\[
    \frac{\|v^{(1)}_+\|}{\|v^{(2)}_+\|}=\frac{\|v^{(1)}_+\|-\sigma_1}{\sigma_2},
\]
thus, rearranging, we find
\[
    \sigma_1
    =
    \left(1-\frac{\sigma_2}{\|v^{(2)}_+\|}\right)\|v^{(1)}_+\|
    =
    \left(1-\frac{1}{\sqrt{\|v^{(2)}_+\|}}\right)\|v^{(1)}_+\|.
\]

Similarly, we let $\tau=(\tau_1,\tau_2) \in \partial \Delta$ be the point on $\partial\Delta$ with coordinates
\[
\tau_2 = \sigma_2 = \sqrt{\|v^{(2)}_+\|} \qquad \text{and} \qquad \tau_1 = -\left(1-\frac{1}{\sqrt{\|v^{(2)}_+\|}}\right)\|v^{(1)}_-\|.
\]

Consider now the rectangle $R$, contained in $\Delta$, whose vertexes are $\sigma$, $\tau$, and their projections on the $x_1$-axis. Clearly, $R$ has sides' lengths given by, respectively, 
\begin{equation*}
\begin{split}
\sigma_1-\tau_1 
&
= 
\left(1-\frac{1}{\sqrt{\|v^{(2)}_+\|}}\right)\|d^{(1)}\|\,, \\
\sigma_2 
& 
=
\sqrt{\|v^{(2)}_+\|}\,.
\end{split}
\end{equation*}
Since of course $\|v^{(2)}_+\|\ge \rho^2$, this concludes the proof of this case.
\begin{figure}[t]
\begin{center}
\begin{tikzpicture}[scale=2]

  \coordinate (O) at (0,0);
  \coordinate (A) at (1.5,1);
  \coordinate (A') at (-1,1);
  \coordinate (Z) at (0,4);

  \coordinate (R1) at (1,1);
  \coordinate (R2) at (1,2);
  \coordinate (R3) at (-0.667,2);
  \coordinate (R4) at (-0.667,1);

    \fill[gray!20] (O) -- (A) -- (Z) -- (A') -- cycle;
    \draw[thick] (O) -- (A) -- (Z) -- (A') -- cycle;
    
  \fill[gray!60] (R1) -- (R2) -- (R3) -- (R4) -- cycle;
  \draw[thick] (R1) -- (R2) -- (R3) -- (R4) -- cycle;

  \draw[dashed, thick] (-1.3,2) -- (2.2,2);


  \draw[->, ultra thin] (0,-0.2) -- (0,4.3) node[below right] {};
  \draw[->, ultra thin] (-1.3,1) -- (2.2,1);
  
     \node[below right] at (A) {$v^{(1)}_+$};
    \foreach \pt/\name in {
    (A)/{v^{(1)}_+},
    (R2)/{\sigma_n},
    (A')/{y'_n},
    (R3)/{Q_n},
    (Z)/{v^{(2)}_+}
    }
  \fill \pt circle (0.03);
  \node[above right] at (R2) {$\sigma=(\sigma_1,\|v^{(2)}_+\|^{1/2})$};
  \node[below left] at (A') {$v^{(1)}_-$};
  \node[above left] at (R3) {$\tau = (\tau_1,\|v^{(2)}_+\|^{1/2})$};
  \node[above right] at (Z) {$v^{(2)}_+$}; 

  \node[above right] at (0.35,1.2) {$R$};
 
\end{tikzpicture}
    \caption{The configuration of \cref{lem:rhomboid}, for each $2$-dimensional rhomboid $\Delta_n$ of the sequence.}
    \label{fig:rhomboid}
\end{center}
\end{figure}
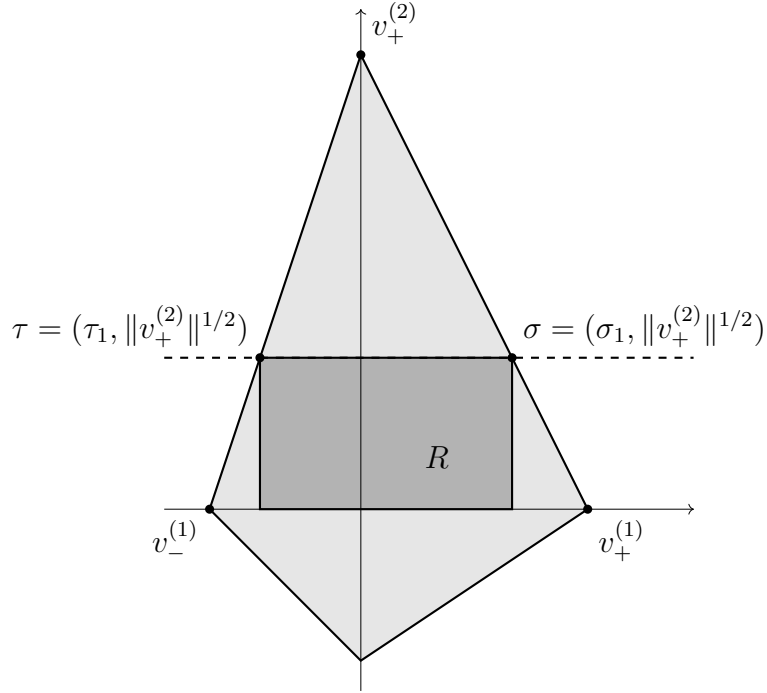

\textbf{The general case.}
Let $\pi\subset\mathbb{R}^N$ be the $2$-dimensional plane identified by $x_2=\dots=x_N$. The intersection $\pi\cap\Delta$ is a $2$-dimensional rhomboid $\Sigma$ whose ``horizontal'' diagonal is $d^{(1)}$. We denote by $d$ the other diagonal.
Let $T$ be 
\[
T = \sup\{\,t\,:\, (0,t,\dots, t)\in \Sigma\,\},
\]
so that $v=(0,T,\dots,T)$ is the vertex with positive coordinates on the diagonal $d$. This vertex is a convex combination of $v^{(1)}_+,\dots,v^{(N)}_+$, that is,
\begin{equation}
\label{eq:convex-combination}
v = \sum_{i=2}^{N}\beta_i v^{(i)}_+, \qquad \text{with } \beta_i\ge 0\quad \text{ and }\quad \sum_{i=2}^{N}\beta_i = 1.
\end{equation}

Therefore, since $v^{(i)}_+$ are pairwise orthogonal,
\[
\|v\| 
= 
\sqrt{\sum_{i=2}^{N}\beta_i^2 \|v^{(i)}_+\|^2}
\ge
\beta_{k} \|v^{(k)}_+\| 
\ge
\beta_{k} \min_i \|v^{(i)}_+\|,
\]
where $k$ is the index such that $\beta_{k} = \max_i \beta_i$. From the last bit of information in~\eqref{eq:convex-combination} we can easily see that $\beta_k\ge1/(N-1)$.
Summing up, we have
\[
\|v\| \ge \frac{1}{N-1}\, \min\{\, \|v^{(i)}_+\|,\,i=2,\dots,N\,\} = \rho^2.
\]
On the other hand, by the same argument used for the case $N=2$, we find two points $\sigma=(\sigma_i)_i$ and $\tau=(\tau_i)_i$ on $\partial \Sigma\subset \partial \Delta$ such that
\begin{equation*}
\begin{split}
& \sigma_i=\sigma_j=\sqrt{\|v\|}, \qquad \forall i,j= 2,\dots, N,
\\
& \sigma_i=\tau_i=\sqrt{\|v\|},  \qquad \,\forall i=2,\dots, N,
\\
& \sigma_1 > 0 > \tau_1,
\\
& \sigma_1-\tau_1 =\left(1-\frac{1}{\sqrt{\|v\|}}\right)\|d^{(1)}\|.
\end{split}
\end{equation*}

We can now consider the $N$-dimensional parallepiped
\[
R = [\tau_1,\sigma_1]\times\left[0, \sqrt{\|v\|}\right]^{N-1} ,
\]
which, by construction, is contained in $\Delta$. Indeed, its vertex 
\[
(\sigma_1,\sqrt{\|v\|},\dots,\sqrt{\|v\|}),
\qquad\text{resp.,}\qquad
(\tau_1,\sqrt{\|v\|},\dots,\sqrt{\|v\|}),
\]
belongs to $\partial \Delta \cap \{x_i\ge 0, \forall i\}$, resp., to $\partial \Delta \cap \{x_1\le 0,\, x_i\ge 0, \forall i \ne1\}$, thus it is a convex combination of the vertexes
\[
\{v^{(i)}_+\}_{i=2}^N \quad  \text{and} \quad v^{(1)}_+, \qquad \text{resp., and } v^{(1)}_-.
\]
Any other vertex can be found from either of these by suitably suppressing some coefficients in their convex combination, and thus it is a sub-convex combination of $N$ distinct vertexes of $\Delta$. In conclusion, every vertex of $R$ belongs to $\Delta$, and thus by convexity we find that $R\subset\Delta$. Replacing $\sqrt{\|v\|}$ with $\rho$ produces an even smaller parallelepiped, whence the conclusion of the proof.
\end{proof}

Now we introduce the notion of a rhomboid-like set in $\R^N$, which will help extend the result proved in \cref{thm:3-dim-reverse-cheeger} to the case $N\ge4$, although with a slightly stronger assumption on the properties of the sequence $(\Omega_n)_n$.

\begin{definition}
\label{def:rhomboid-like}
    Let $\Omega\in\BK_N$, and let $q^{(i)},p^{(i)},\nu^{(i)}$ be defined as in \cref{def:p-widths} and \cref{prop:realization_principal_width} for every $i=1,\dots,N$.
    We say that $\Omega$ is \emph{rhomboid-like} if the points $q^{(i)},p^{(i)}$ belong to the plane $\vspan \{ \nu^{(1)},\nu^{(i)}\}$ for every $i=2,\dots,N$.
\end{definition}

\begin{theorem}
\label{thm:rhomboid-like}
    Let $\Omega\in\BK_N$ be a rhomboid-like set, having minimal width $w^{(1)}=w^{(1)}(\Omega)$ and relative width $w^{(\mathrm{rel})}=w^{(\mathrm{rel})}(\Omega)$.
    Then, $\Omega$ contains an isometric copy of the $N$-dimensional parallelepiped
    \[
        \left[0,w^{(1)}\left(1-2\sqrt{(N-1)w^{(\mathrm{rel})}}\right)\right]\times\left[0,w^{(1)}/\left(2\sqrt{(N-1)w^{(\mathrm{rel})}}\right)\right]^{N-1}
    \]
\end{theorem}
\begin{proof}
    As usual, we can assume without loss of generality that $\nu^{(i)}=e_i$ for every $i=1,\dots,N$. Moreover, we assume that $q^{(1)}=(-1/2,0,\dots,0)$ and $p^{(1)}=(1/2,0,\dots,0)$, which implies $w^{(1)}=1$, and that
    \begin{equation}
    \label{eq:rhomboid-reflection}
        \forall i=2,\dots,N, \qquad\langle p^{(i)},e_i \rangle\ge\frac{w^{(i)}}{2}\,.
    \end{equation}
    
    For every $i=1,\dots,N$, we denote by $\Sigma_i\in\BK_2$ the intersection between $\Omega$ and the plane $\pi_i=\vspan(e_1,e_i)$.
    By the assumption that $\Omega$ is rhomboid-like, every $\Sigma_i$ is such that $w^{(i)}(\Sigma_i)=w^{(i)}(\Omega)$.
    Hence, if we define
    \begin{equation}
    \label{eq:unilateral-width}
        w^{(i)}_-(\Omega)=\langle q^{(i)},e_i\rangle \qquad \text{and} \qquad  w^{(i)}_+(\Omega)=\langle p^{(i)},e_i \rangle \qquad \forall i=2,\dots,N,
    \end{equation}
    then $\Sigma_i$ contains the points
    \[
        \tilde{q}^{(i)}=\frac{w^{(i)}_-(\Omega)}{2}e_i \qquad \text{and} \qquad \tilde{p}^{(i)}=\frac{w^{(i)}_+(\Omega)}{2}e_i \qquad \forall i=2,\dots,N.
    \]
    This implies that $\Omega$ contains the convex hull of the $2N$ points
    \[
    q^{(1)},p^{(1)},\tilde{q}^{(2)},\tilde{p}^{(2)},\dots,\tilde{q}^{(N)},\tilde{p}^{(N)},
    \]
    which is an $N$-dimensional rhomboid $\Delta$ having these points as vertexes.
    Hence, by \cref{lem:rhomboid}, $\Omega$ contains an isometric copy of the parallelepiped
    \[
        \left[0,\left(1-\frac{1}{\sqrt{\|v\|}}\right)\right]\times\left[0,\sqrt{\|v\|}\right]^{N-1}
    \]
    where $v\in\R^N$ is such that
    \[
        \|v\| \ge \frac{1}{N-1}\, \min_{2\le i\le N}\{\, \|v_+^{(i)}\|\,\}\,,
    \]
    and the points $v_+^{(i)}$ are defined as in \cref{defin:rhomboid}.
    Here we have $v_+^{(i)}=\tilde{p}^{(i)}$ for every $i=2,\dots,N$, and thus, by combining \eqref{eq:rhomboid-reflection} and \eqref{eq:unilateral-width}, we get
    \[
         \|v_+^{(i)}\|=\frac{w_+^{(i)}(\Omega)}{2}\ge\frac{w^{(i)}}{4}\ge\frac{w^{(2)}}{4}=\frac{1}{4w^{(\mathrm{rel})}},
    \]
    and so it follows that
    \[
        \sqrt{\|v\|}\ge\frac{1}{2\sqrt{(N-1)w^{(\mathrm{rel})}}}
    \]
    which concludes the proof.
\end{proof}

\begin{theorem}
    Every sequence of rhomboid-like sets $(\Omega_n)_n\subset\BK_N$ with vanishing relative widths is maximizing of the reverse Cheeger inequality.
\end{theorem}

The proof is completely analogous to that of \cref{thm:3-dim-reverse-cheeger}, using \cref{thm:rhomboid-like} in place of \cref{thm:smallrelwidth}.

\section*{Acknowledgments} 

G.~S.~wishes to thank Andrea Colesanti for helpful discussions on Convex Geometry. The third author is member of INdAM--GNAMPA and is partially supported by the INdAM--GNAMPA Project, codice CUP \#E5324001950001\#, ``Disuguaglianze funzionali di tipo geometrico e spettrale''.

N.C. is funded by the Research Foundation – Flanders (FWO) via the Odysseus II programme no.~G0DBZ23N.

\noindent


\bibliographystyle{plainurl} 
\bibliography{bib_buser.bib} 

\end{document}